\documentclass[reqno, 12pt]{amsart}
\usepackage{amsmath,amsfonts,amsthm,amscd,upref,amstext}
\usepackage{amssymb}
\usepackage{mathtools}
\usepackage{xcolor}
\usepackage{fullpage}
\usepackage{lmodern}

\usepackage[utf8]{inputenc}
\usepackage[T1]{fontenc}
\usepackage{amsmath,amssymb,amsthm,amsfonts,amscd,amstext}
\usepackage{mathtools,mathrsfs}
\usepackage{enumitem}
\usepackage{lmodern}
\usepackage{hyperref}
\hypersetup{
  pdftitle={Generalized Mixed Sequence Multiplicity},
  pdfauthor={Victor Hugo Jorge Perez and Thiago Freitas}
}

\newtheorem{theorem}{Theorem}[section]
\newtheorem{lemma}[theorem]{Lemma}
\newtheorem{proposition}[theorem]{Proposition}
\newtheorem{corollary}[theorem]{Corollary}

\theoremstyle{definition}
\newtheorem{definition}[theorem]{Definition}

\theoremstyle{remark}
\newtheorem{remark}[theorem]{Remark}

\title{Generalized Mixed Sequence Multiplicity}
\author{T. H. Freitas}
\address{Universidade Tecnol\'ogica Federal do Paran\'a, 85053-525, Guarapuava-PR, Brazil}
\email{freitas.thf@gmail.com}

\author{V. H. Jorge P\'erez}
\address{Universidade de S{\~a}o Paulo -- ICMC, 13560-970, S{\~a}o Carlos-SP, Brazil}
\email{vhjperez@icmc.usp.br}
\date{}

\begin{document}

\begin{abstract}

Let $(R,\mathfrak m,k)$ be a Noetherian local ring, let $M$ be a nonzero finitely generated $R$-module, and let $J_1,\ldots,J_s$ be arbitrary proper ideals of positive height on $M$. We introduce the notion of directional mixed multiplicities, obtained from the multigraded quotients $J^{\mathbf u}M/J^{\mathbf u+\mathbf e_j}M$ by retaining one distinguished direction $j$. These invariants refine the mixed multiplicity sequences of arbitrary ideals while remaining compatible with the two classical extremes: for one ideal they recover the positive components of the Achilles--Manaresi multiplicity sequence, and for $\mathfrak m$-primary ideals their top-degree terms recover the usual mixed multiplicities individually. We establish additivity, an associativity formula, superficial reduction, and an exact decomposition of the diagonal coefficients into directional components. Over an infinite residue field, every positive directional coefficient is realized as an eventual length by a sufficiently general directional $(FC)$-system. We then identify an additional synchronization condition under which the directional coefficient becomes the Hilbert--Samuel multiplicity of a parameter ideal, yielding a Rees-type theorem and recovering the classical theorem in the $\mathfrak m$-primary case. In addition, we show a multihomogeneous product formula for $J_1^{n_1}\cdots J_s^{n_s}$, which specializes to the classical Bhattacharya--Rees--Teissier formula.

\end{abstract}


\maketitle

\section{Introduction}
\label{sec:introduction}

Mixed multiplicities are a multigraded extension of Hilbert--Samuel
multiplicity and have long provided a bridge between local algebra and
intersection-theoretic questions.  Their origins go back to
Bhattacharya's study of the Hilbert function of two $\mathfrak m$-primary
ideals \cite{Bhattacharya} and to the work of Risler and Teissier on
several primary ideals and intersection theory \cite{Teissier}.  A
fundamental step was taken by Rees, who introduced joint reductions and
showed that classical mixed multiplicities can be realized as ordinary
Hilbert--Samuel multiplicities of ideals generated by suitable joint
reductions \cite{Rees} (see also \cite[Chapter~17]{HunekeSwanson} for further details).  The
theory has since interacted with integral dependence, reduction theory,
multigraded algebras, and convex geometry (see for instance 
\cite{TrungVerma}).

The situation changes substantially when the ideals are not
$\mathfrak m$-primary.  The usual Hilbert--Samuel function no longer has
finite length, so one must replace the classical construction by a more
refined multigraded object. In this sense, Achilles and Manaresi introduced a
multiplicity sequence for an arbitrary ideal by means of a bigraded
construction \cite{AchillesManaresi}.  This sequence retains many of the
features of ordinary multiplicity: its additivity and associativity
properties were developed in \cite{CallejasBedregalJorgePerezAdditivity},
and recently it was shown to detect integral dependence under
natural hypotheses \cite{PoliniTrungUlrichValidashti}.  For several
arbitrary ideals, Callejas--Bedregal and Jorge P\'erez introduced mixed
multiplicity sequences using the diagonal quotients $  \frac{J^{\mathbf u}M}{J^{\mathbf u+\mathbf 1}M},$
and obtained, among other consequences, a product formula extending the
one-ideal theory \cite{CallejasBedregalJorgePerez}.

There is, however, a feature of the diagonal construction that becomes
visible already in the classical $\mathfrak m$-primary case.  A diagonal
coefficient is not a single Rees--Teissier mixed multiplicity; rather, it
is a sum of adjacent mixed multiplicities corresponding to the possible
choices of the ideal that contributes the last increment.  Thus the
symmetry of the diagonal quotient comes at the price of losing the
individual directional information.  This leads to the question that
motivates the present paper: can one refine the mixed multiplicity
sequence of several arbitrary ideals so that the distinguished ideal is
remembered, while still recovering both the Achilles--Manaresi sequence
for one ideal and the individual classical mixed multiplicities for
$\mathfrak m$-primary ideals?

Our starting point is to replace the diagonal quotient by the
$j$-directional quotient $\frac{J^{\mathbf u}M}{J^{\mathbf u+\mathbf e_j}M}.$
Proposition~\ref{prop:directional-polynomiality} shows that the associated
cumulative multigraded Hilbert function is eventually polynomial, with
top part of total degree at most $D-1$, where $D=\dim_R M$.  Its normalized
top-degree coefficients are the directional mixed multiplicities of
Definition~\ref{def:directional-mixed-multiplicities}.  The first basic
compatibility results show that this is not merely a formal refinement.
Proposition~\ref{prop:one-ideal-directional} proves that, for a single
ideal $I$, one has
\[
   c^{(1)}(I^{[i]};M)=c_i(I;M),
\]
so the construction recovers the positive components of the
Achilles--Manaresi multiplicity sequence.  At the other classical
extreme, Theorem~\ref{thm:directional-primary-compatibility} gives that
for $\mathfrak m$-primary ideals all lower directional coefficients
vanish and, whenever $d_1+\cdots+d_s=D$ and $d_j>0$,
\[
 c^{(j)}\bigl(J_1^{[d_1]},\ldots,J_s^{[d_s]};M\bigr)
 =e\bigl(J_1^{[d_1]},\ldots,J_s^{[d_s]};M\bigr).
\]
Thus the directional theory interpolates term by term between the
one-ideal multiplicity sequence and the classical Rees--Teissier theory.

The structural properties of the new coefficients parallel the familiar
properties of multiplicity, but they must be proved while keeping track
of the distinguished direction.  Lemma~\ref{lem:directional-additivity-exact-sequences}
establishes additivity on short exact sequences, and
Theorem~\ref{thm:directional-additivity-reduction-formula} gives an
associativity formula over the top-dimensional minimal components of
$M$.  Theorem~\ref{thm:directional-superficial-reduction} shows that a
suitable directional superficial element removes one occurrence of the
corresponding ideal from the type.  Most importantly for comparison with
the existing theory of several arbitrary ideals,
Proposition~\ref{prop:diagonal-directional-sum} provides that every diagonal
coefficient of Callejas--Bedregal and Jorge P\'erez is exactly the sum of
its directional components.  

The next issue is positivity and realization.  For arbitrary ideals,
D.~Q.~Vi\^et introduced weak-$(FC)$- and $(FC)$-sequences as a way of
controlling mixed multiplicities and reductions \cite{Viet2000}. These ideas was further developed through filter-regular sequences,
Rees superficial sequences, and joint reductions in
\cite{ManhViet,ThanhViet, VietDinhThanh}.  Our directional module is
different from the multigraded modules occurring in those theories, so
the classical definitions cannot simply be imported.  We therefore
define directional weak-$(FC)$-elements and directional $(FC)$-systems
on the multigraded module attached to the quotient
$J^{\mathbf u}M/J^{\mathbf u+\mathbf e_j}M$.  If the residue field is
infinite, Theorem~\ref{thm:directional-FC-existence} shows that every
positive directional mixed multiplicity admits a sufficiently general
such system and, moreover, is computed by an eventual constant length.
This gives an intrinsic realization of every positive directional
coefficient.

Turning this intrinsic length into a Hilbert--Samuel multiplicity is
more delicate.  In the classical $\mathfrak m$-primary setting, Rees'
theorem identifies mixed multiplicities with multiplicities of joint
reductions \cite{HunekeSwanson,Rees}.  For arbitrary ideals, Vi\^et and
subsequent authors obtained related Rees-type results under additional
hypotheses (see \cite{ThanhViet, VietDinhThanh}).  For the directional
invariants considered here, however, the ordinary joint-reduction
equation does not by itself remember the distinguished grading.  The
example in Remark~\ref{rem:ordinary-joint-reduction-not-enough} shows
that an ordinary joint reduction may generate an ideal of positive
Hilbert--Samuel multiplicity while the corresponding directional mixed
multiplicity is zero.  This is why an additional compatibility is
necessary.  Definition~\ref{def:synchronized-Rees-realization} isolates
a synchronization condition between the terminal directional quotient
and a one-dimensional parameter module.  Under this condition,
Theorem~\ref{thm:directional-Rees-FC} gives the Rees-type formula
\[
 c_R^{(j)}\bigl(J_1^{[d_1]},\ldots,J_s^{[d_s]};M\bigr)
 =e_R(Q;M),
\]
where $Q$ is the parameter ideal generated by the synchronized
realization.  Also,
Corollary~\ref{cor:classical-Rees-from-directional-FC} recovers Rees'
classical mixed-multiplicity theorem.

A further advantage of retaining the direction is that product formulas
can be refined before the directional information is summed away.
Theorem~\ref{thm:directional-product-formula-powers} proves that, for
$\mathbf n=(n_1,\ldots,n_s)\in\mathbb N_{>0}^s$ and
$I(\mathbf n)=J_1^{n_1}\cdots J_s^{n_s}$,
\[
\begin{aligned}
c_i\bigl(I(\mathbf n);M\bigr)
={}&\sum_{d_1+\cdots+d_s=i}
 \ \sum_{\substack{1\leq j\leq s\\ d_j>0}}
 \frac{(i-1)!}{d_1!\cdots(d_j-1)!\cdots d_s!}\,
 c^{(j)}\bigl(J_1^{[d_1]},\ldots,J_s^{[d_s]};M\bigr)
 n_1^{d_1}\cdots n_s^{d_s}.
\end{aligned}
\]
For $n_1=\cdots=n_s=1$, summing over the distinguished directions
recovers the product formula of
Callejas--Bedregal and Jorge P\'erez (Remark~\ref{rem:comparison-CBJP-product-formula}).  When the ideals are
$\mathfrak m$-primary, Corollary~\ref{cor:classical-product-formula-from-directional}
specializes the same identity to the classical
Bhattacharya--Rees--Teissier formula.  Thus the product theorem exhibits,
in a single formula, how the arbitrary-ideal multiplicity sequence and
the classical mixed multiplicities fit inside the directional theory.

The paper is organized as follows.
Section~\ref{sec:directional-mixed-multiplicities} introduces the
directional multiform modules and directional mixed multiplicities, and
establishes their compatibility with the one-ideal and
$\mathfrak m$-primary theories.  Section~\ref{sec:structural-properties}
shows additivity, associativity, superficial reduction, and the exact
decomposition of the diagonal coefficients.  In
Section~\ref{sec:directional-FC-Rees} we develop directional
$(FC)$-systems, prove the length realization theorem, and establish the
synchronized Rees-type formula together with its classical
specialization.  The last section proves the
multihomogeneous product formula and derives the classical product
formula as a consequence.

\section{Directional mixed multiplicities of arbitrary ideals}
\label{sec:directional-mixed-multiplicities}

Throughout this section, let \((R,\mathfrak m,k)\) be a Noetherian
local ring and let \(M\) be a nonzero finitely generated \(R\)-module
of dimension
\(
 D:=\dim_R M\geq1.
\)
We first recall the classical mixed multiplicities of
\(\mathfrak m\)-primary ideals on \(M\) and the multiplicity sequence
of an arbitrary ideal on \(M\). Also, we introduce a module refinement of the construction of Callejas--Bedregal and
Jorge P\'erez \cite[Section~5]{CallejasBedregalJorgePerez}. The proposed refinement is designed to have the following two
properties. For one ideal, it recovers the individual terms of the
multiplicity sequence of \(I\) on \(M\). For \(\mathfrak m\)-primary
ideals, it recovers each classical mixed multiplicity of \(M\)
separately, rather than a sum of adjacent mixed multiplicities.

Set \(\overline R=R/\operatorname{Ann}_R(M)\). For an ideal
\(I\subseteq R\), put
\(
 \operatorname{ht}_M(I)
 :=\operatorname{ht}_{\overline R}
\bigl((I+\operatorname{Ann}_R(M))/\operatorname{Ann}_R(M)\bigr).
\)
Thus \(\operatorname{ht}_M(I)>0\) means that \(I\) has positive height
on \(\operatorname{Supp}_R(M)\). This is the meaning of ``positive
height'' used below.

\subsection*{Classical mixed multiplicities of
\texorpdfstring{\(\mathfrak m\)}{m}-primary ideals}

Let \(J_1,\ldots,J_s\) be \(\mathfrak m\)-primary ideals of \(R\). For
\(\mathbf u=(u_1,\ldots,u_s)\in\mathbb N^s\), write
\(
 J^{\mathbf u}:=J_1^{u_1}\cdots J_s^{u_s}.
\)
The Bhattacharya function associated with \(J_1,\ldots,J_s\) and
\(M\) is
\[
 B_{\mathbf J,M}(\mathbf u)
 :=\lambda_R\bigl(M/J^{\mathbf u}M\bigr).
\]
For \(\mathbf u\gg\mathbf0\), this function agrees with a polynomial
\(Q_{\mathbf J,M}(\mathbf u)\) of total degree \(D\). The ring case is
due
to Bhattacharya for two ideals and to Risler--Teissier in the general
case (see \cite{Bhattacharya,Teissier},
\cite[Section~4]{CallejasBedregalJorgePerez}, and
\cite[Section~17.4]{HunekeSwanson}). We use the classical repetition notation. Thus, if
\(\mathbf d=(d_1,\ldots,d_s)\in\mathbb N^s\) and
\(|\mathbf d|=d_1+\cdots+d_s=D\), the homogeneous component of degree
\(D\) of the Bhattacharya polynomial is
\begin{equation}
 Q_{\mathbf J,M,D}(\mathbf u)
 =\sum_{d_1+\cdots+d_s=D}
 \frac{e\bigl(J_1^{[d_1]},\ldots,J_s^{[d_s]};M\bigr)}
 {d_1!\cdots d_s!}
 u_1^{d_1}\cdots u_s^{d_s}.
 \label{eq:Bhattacharya-repetition}
\end{equation}
The nonnegative integer
\(
 e\bigl(J_1^{[d_1]},\ldots,J_s^{[d_s]};M\bigr)
\)
is the classical mixed multiplicity of \(M\), with respect to
\(J_1,\ldots,J_s\), of type \((d_1,\ldots,d_s)\). The symbol
\(J_\ell^{[d_\ell]}\) indicates that
\(J_\ell\) is repeated \(d_\ell\) times; it does not denote an
ordinary power or a Frobenius power. If \(d_\ell=1\), we write
\(J_\ell\) in place of \(J_\ell^{[1]}\). Equivalently,
\[
 e\bigl(J_1^{[d_1]},\ldots,J_s^{[d_s]};M\bigr)
 =e\bigl(
 \underbrace{J_1,\ldots,J_1}_{d_1\text{ times}},\ldots,
 \underbrace{J_s,\ldots,J_s}_{d_s\text{ times}};M
 \bigr).
\]

The pure types recover the Hilbert--Samuel multiplicities:
\begin{equation}
 e\bigl(J_1^{[0]},\ldots,J_j^{[D]},\ldots,J_s^{[0]};M\bigr)
 =e\bigl(J_j^{[D]};M\bigr)=e(J_j;M).
 \label{eq:pure-classical-type}
\end{equation}
In addition, the classical product formula is given by
\begin{equation}
 e(J_1\cdots J_s;M)
 =\sum_{d_1+\cdots+d_s=D}
 \binom{D}{d_1,\ldots,d_s}
 e\bigl(J_1^{[d_1]},\ldots,J_s^{[d_s]};M\bigr),
 \label{eq:classical-product-formula}
\end{equation}
where
\(
 \binom{D}{d_1,\ldots,d_s}=D!/(d_1!\cdots d_s!)
\)
(see  \cite{Teissier},
\cite[Theorem~4.1]{CallejasBedregalJorgePerez}, and
\cite[Sections~17.4--17.5]{HunekeSwanson}).
The passage from rings to modules is governed by the associativity
formula. If
\(
 \operatorname{Assh}_R(M)
 :=\{\mathfrak p\in\operatorname{Ass}_R(M)
       :\dim(R/\mathfrak p)=D\},\)
then, for every \(d_1+\cdots+d_s=D\),
\begin{equation}
 e\bigl(J_1^{[d_1]},\ldots,J_s^{[d_s]};M\bigr)
 =\sum_{\mathfrak p\in\operatorname{Assh}_R(M)}
 \lambda_{R_{\mathfrak p}}(M_{\mathfrak p})
 e\bigl(J_1^{[d_1]},\ldots,J_s^{[d_s]};R/\mathfrak p\bigr).
 \label{eq:classical-module-associativity}
\end{equation}
Thus the classical mixed multiplicities of a module depend, in top
degree, only on its top-dimensional components and their generic
lengths (see \cite[Section~17.4]{HunekeSwanson}).

\subsection*{The multiplicity sequence}

Let \(I\subsetneq R\) be an arbitrary ideal and put
\[
 G_I(M):=\operatorname{gr}_I(M)
 =\bigoplus_{b\geq0}I^bM/I^{b+1}M.
\]
The doubly associated graded module
\(
 \mathscr G(I;M):=\operatorname{gr}_{\mathfrak m}(G_I(M))
\) is a finitely generated bigraded module over the corresponding
standard bigraded \(k\)-algebra, with components
\[
 \mathscr G(I;M)_{a,b}
 =\frac{\mathfrak m^aI^bM+I^{b+1}M}
 {\mathfrak m^{a+1}I^bM+I^{b+1}M}.
\]
Consider the doubly cumulative Hilbert function
\[
 h_{I,M}(r,u):=\sum_{a=0}^{r}\sum_{b=0}^{u}
 \lambda_R\bigl(\mathscr G(I;M)_{a,b}\bigr).
\]
For \(r,u\gg0\), this function agrees with a polynomial of total
degree at most \(D\), whose homogeneous component of degree \(D\) is
\begin{equation}
 \sum_{i=0}^{D}\frac{c_i(I;M)}{(D-i)!\,i!}r^{D-i}u^i.
 \label{eq:multiplicity-sequence-polynomial}
\end{equation}
The sequence \(c_0(I;M),\ldots,c_D(I;M)\) is the multiplicity
sequence of \(I\) on \(M\); it is the module version of the
Achilles--Manaresi sequence (\cite[Definition~2.2]{AchillesManaresi} and
\cite[Section~2]{PoliniTrungUlrichValidashti}).
We use the codimension indexing of
\cite{PoliniTrungUlrichValidashti}, which reverses the original
dimension indexing of Achilles--Manaresi. If
\(c_q^{\mathrm{AM}}(I;M)\) denotes the original invariant, then
\(
 c_i(I;M)=c_{D-i}^{\mathrm{AM}}(I;M).
\)
In particular, if \(I\) is \(\mathfrak m\)-primary, then
\(
 c_i(I;M)=0\) for  $i<D$, and $c_D(I;M)=e(I;M)$.

\subsection*{Directional multiform modules}

We now let \(J_1,\ldots,J_s\) be arbitrary proper ideals satisfying
\(\operatorname{ht}_M(J_\ell)>0\) for every \(\ell\). Consider the
multi-Rees algebra
\[
 \mathcal R(\mathbf J)
 :=R[J_1T_1,\ldots,J_sT_s]
 =\bigoplus_{\mathbf u\in\mathbb N^s}
 J^{\mathbf u}\mathbf T^{\mathbf u},
\]
where \(\mathbf T^{\mathbf u}=T_1^{u_1}\cdots T_s^{u_s}\).
The multi-Rees module of \(\mathbf J\) on \(M\) is
\(
 \mathcal R(\mathbf J;M)
 :=\bigoplus_{\mathbf u\in\mathbb N^s}
 J^{\mathbf u}M\,\mathbf T^{\mathbf u}.
\)
It is a finitely generated \(\mathbb N^s\)-graded module over
\(\mathcal R(\mathbf J)\).
For \(j\in\{1,\ldots,s\}\), define the \(j\)-directional multiform
module by
\[
 \mathscr S^{(j)}(\mathbf J;M)
 :=\frac{\mathcal R(\mathbf J;M)}
 {J_j\mathcal R(\mathbf J;M)}
 =\bigoplus_{\mathbf u\in\mathbb N^s}
 \frac{J^{\mathbf u}M}{J^{\mathbf u+\mathbf e_j}M}
 \mathbf T^{\mathbf u}.
\]
Its associated graded module with respect to \(\mathfrak m\) is
\[
 \mathscr G^{(j)}(\mathbf J;M)
 :=\operatorname{gr}_{\mathfrak m}\mathscr S^{(j)}(\mathbf J;M)
 =\bigoplus_{\substack{r\geq0\\\mathbf u\in\mathbb N^s}}
 \frac{\mathfrak m^rJ^{\mathbf u}M+J^{\mathbf u+\mathbf e_j}M}
 {\mathfrak m^{r+1}J^{\mathbf u}M+J^{\mathbf u+\mathbf e_j}M}.
\]
This is a finitely generated \(\mathbb N^{s+1}\)-graded module over
the corresponding standard \(k\)-algebra. Its cumulative Hilbert
function in the \(\mathfrak m\)-direction is given by
\begin{align}
 H_{\mathbf J,M}^{(j)}(r,\mathbf u)
 &:=\sum_{q=0}^{r}
 \lambda_R\bigl(\mathscr G^{(j)}(\mathbf J;M)_{q,\mathbf u}\bigr) =\lambda_R\left(
 \frac{J^{\mathbf u}M}
 {\mathfrak m^{r+1}J^{\mathbf u}M+J^{\mathbf u+\mathbf e_j}M}
 \right).
 \label{eq:directional-Hilbert-function}
\end{align}

We start with a key result for the rest of the paper. 
\begin{proposition}
\label{prop:directional-polynomiality}
Let \((R,\mathfrak m,k)\) be a Noetherian local ring, let \(M\) be a nonzero finitely generated \(R\)-module of dimension \(D\), and let \(J_1,\ldots,J_s\) be proper ideals such that \(\operatorname{ht}_M(J_\ell)>0\) for every \(\ell\). Fix \(j\in\{1,\ldots,s\}\), and put \(E^{(j)}:=\mathscr G^{(j)}(\mathbf J;M)\). Then the function \(H_{\mathbf J,M}^{(j)}(r,\mathbf u) = \sum_{q=0}^{r} \lambda_R\bigl(E^{(j)}_{q,\mathbf u}\bigr)\) agrees, for \(r,\mathbf u\gg0\), with a polynomial \(P_{\mathbf J,M}^{(j)}(r,\mathbf u)\in \mathbb Q[r,u_1,\ldots,u_s]\) of total degree at most \(D-1\). If \(P_{\mathbf J,M,D-1}^{(j)}\) denotes its homogeneous component of degree \(D-1\), with the convention that this component is zero when \(\deg P_{\mathbf J,M}^{(j)}<D-1\), then there are uniquely determined nonnegative integers \(c^{(j)}\bigl( J_1^{[d_1]},\ldots,J_s^{[d_s]};M \bigr)\), \(1\leq|\mathbf d|\leq D\), \(d_j>0\), such that
\begin{equation}
\begin{aligned}
P_{\mathbf J,M,D-1}^{(j)}(r,\mathbf u)
={}&\sum_{\substack{\mathbf d\in\mathbb N^s\\
                     1\leq|\mathbf d|\leq D,\ d_j>0}}
 \frac{c^{(j)}\bigl(J_1^{[d_1]},\ldots,J_s^{[d_s]};M\bigr)}
 {(D-|\mathbf d|)!\,(d_j-1)!\prod_{\ell\neq j}d_\ell!} \cdot
 r^{D-|\mathbf d|}u_j^{d_j-1}
 \prod_{\ell\neq j}u_\ell^{d_\ell}.
\end{aligned}
\label{eq:proved-directional-polynomial-repetition}
\end{equation}
\end{proposition}

\begin{proof}
Set \(A:=\mathcal R(\mathbf J)\) and \(\mathcal M:=\mathcal R(\mathbf J;M)\). The module \(\mathcal M\) is a finitely generated \(\mathbb N^s\)-graded \(A\)-module, and hence \(\mathscr S^{(j)}(\mathbf J;M) =\mathcal M/J_j\mathcal M\) is also finitely generated. Let \(B:=\operatorname{gr}_{\mathfrak mA}(A)\). This is a standard \(\mathbb N^{s+1}\)-graded \(k\)-algebra, and \(E^{(j)} =\operatorname{gr}_{\mathfrak mA} \bigl(\mathscr S^{(j)}(\mathbf J;M)\bigr)\) is a finitely generated \(\mathbb N^{s+1}\)-graded \(B\)-module. Its component of degree \((q,\mathbf u)\) is precisely
\[
E^{(j)}_{q,\mathbf u} = \frac{\mathfrak m^qJ^{\mathbf u}M+ J^{\mathbf u+\mathbf e_j}M} {\mathfrak m^{q+1}J^{\mathbf u}M+ J^{\mathbf u+\mathbf e_j}M}.
\]
Each such component is a finite-dimensional \(k\)-vector space, and its \(k\)-dimension equals its \(R\)-length.

Next, we  record the dimension estimate responsible for the bound \(D-1\). Put \(\overline R:=R/\operatorname{Ann}_R(M)\). The multi-Rees module \(\mathcal M\) is finite over the multi-Rees algebra of the images of the \(J_\ell\) in \(\overline R\). Consider the diagonal multiform module
\[
\mathscr S^\Delta(\mathbf J;M) := \frac{\mathcal M}{(J_1\cdots J_s)\mathcal M} = \bigoplus_{\mathbf u\in\mathbb N^s} \frac{J^{\mathbf u}M}{J^{\mathbf u+\mathbf1}M} \mathbf T^{\mathbf u}.
\]
It is finite over the diagonal multiform algebra of the images of the \(J_\ell\) in \(\overline R\). Notice that \(\operatorname{ht}_{\overline R} (\overline J_1\cdots\overline J_s)>0\): otherwise the product, and hence one of its factors, would be contained in a minimal prime of \(\overline R\), contrary to the hypotheses. Therefore that diagonal multiform algebra has dimension \(D+s-1\) (\cite[Section~5, before Definition~5.1]{CallejasBedregalJorgePerez}). Hence \(\dim \mathscr S^\Delta(\mathbf J;M)\leq D+s-1\). Since \(J^{\mathbf u+\mathbf1}M\subseteq J^{\mathbf u+\mathbf e_j}M\), degree by degree there is a natural homogeneous surjection \(\mathscr S^\Delta(\mathbf J;M) \twoheadrightarrow \mathscr S^{(j)}(\mathbf J;M)\). Consequently,
\begin{equation}
\dim_A\mathscr S^{(j)}(\mathbf J;M) \leq D+s-1.
\label{eq:dimension-directional-multiform}
\end{equation}
Passing to an associated graded module does not increase dimension (this follows, for instance, from the extended Rees-module construction). Thus
\begin{equation}
\dim_B E^{(j)}\leq D+s-1.
\label{eq:dimension-directional-associated-graded}
\end{equation}

To incorporate the summation in the \(\mathfrak m\)-degree, introduce an indeterminate \(Z\) of multidegree \(\deg Z=(1,\mathbf0)\) and define \(\widetilde B:=B[Z]\) and \(\widetilde E^{(j)}:=E^{(j)}\otimes_k k[Z]\). This is a finitely generated \(\mathbb N^{s+1}\)-graded \(\widetilde B\)-module, and
\(
\widetilde E^{(j)}_{r,\mathbf u} = \bigoplus_{q=0}^{r} E^{(j)}_{q,\mathbf u}Z^{r-q}.
\) It follows that
\begin{equation}
\dim_k\widetilde E^{(j)}_{r,\mathbf u} = \sum_{q=0}^{r}\dim_kE^{(j)}_{q,\mathbf u} = H_{\mathbf J,M}^{(j)}(r,\mathbf u).
\label{eq:cumulative-as-ordinary-hilbert-function}
\end{equation}
Furthermore, by \eqref{eq:dimension-directional-associated-graded}, \(\dim_{\widetilde B}\widetilde E^{(j)} \leq D+s\).
The multigraded Hilbert--Serre theorem, in its module form \cite[Theorem~4.1]{HerrmannHyryRibbeTang}, now shows that the right-hand side of \eqref{eq:cumulative-as-ordinary-hilbert-function} agrees, for \(r,\mathbf u\gg0\), with a polynomial. Its total degree is \(\dim\operatorname{Supp}_{++}(\widetilde E^{(j)})\), and hence is at most \(\dim\widetilde E^{(j)}-(s+1) \leq (D+s)-(s+1)=D-1\). This proves the asserted polynomiality and degree bound. The same multigraded Hilbert theorem shows that the coefficients of the highest-degree monomials, when normalized by factorials, are nonnegative integers. This is exactly the normalization used for mixed multiplicities of multigraded structures (see \cite[Section~2, especially Definition~2.1]{CallejasBedregalJorgePerez}).

It remains only to justify the indexing in \eqref{eq:proved-directional-polynomial-repetition}. Every monomial of total degree \(D-1\) has a unique expression \(r^{a_0}u_1^{a_1}\cdots u_s^{a_s}\) with \(a_0+\cdots+a_s=D-1\) and \(a_0,\ldots,a_s\in\mathbb N\). Define \(d_j:=a_j+1\) and \(d_\ell:=a_\ell\) for \(\ell\neq j\). Then \(d_j>0\) and \(|\mathbf d|=1+\sum_{\ell=1}^{s}a_\ell =D-a_0\), so \(1\leq|\mathbf d|\leq D\), and \(a_0=D-|\mathbf d|\), \(a_j=d_j-1\), \(a_\ell=d_\ell\) for \(\ell\neq j\). Conversely, every \(\mathbf d\) satisfying \(1\leq|\mathbf d|\leq D\) and \(d_j>0\) determines exactly one such monomial. Therefore this change of indices is a bijection, and the standard factorial normalization \(1/(a_0!a_1!\cdots a_s!)\) becomes
\[
\frac{1} {(D-|\mathbf d|)!\,(d_j-1)! \prod_{\ell\neq j}d_\ell!}.
\]
This proves both the uniqueness of the coefficients and formula \eqref{eq:proved-directional-polynomial-repetition}, as desired.
\end{proof}

\begin{definition}
\label{def:directional-mixed-multiplicities}
Let \(J_1,\ldots,J_s\) be proper ideals with
\(\operatorname{ht}_M(J_\ell)>0\) for every \(\ell\). Let
\(\mathbf d=(d_1,\ldots,d_s)\in\mathbb N^s\) satisfy
\(1\leq|\mathbf d|\leq D\), and choose
\(j\in\{1,\ldots,s\}\) such that \(d_j>0\). The coefficient
\(
c^{(j)}\bigl(J_1^{[d_1]},\ldots,J_s^{[d_s]};M\bigr)
\)
in Proposition~\ref{prop:directional-polynomiality} is called the
{\it \(j\)-directional mixed multiplicity} of \(M\), with respect to
\(J_1,\ldots,J_s\), of type \((d_1,\ldots,d_s)\). 

If \(d_i=1\), we
write \(J_i\) in place of \(J_i^{[1]}\). When convenient, we also write
\[
c^{(j)}\bigl(J_1^{[d_1]},\ldots,J_s^{[d_s]};M\bigr)
=c^{(j)}\bigl(
 \underbrace{J_1,\ldots,J_1}_{d_1\text{ times}},\ldots,
 \underbrace{J_s,\ldots,J_s}_{d_s\text{ times}};M\bigr).
\]
\end{definition}
As in the classical notation, \(J_\ell^{[d_\ell]}\) records
repetition, not exponentiation. If \(d_\ell=1\), the brackets are
omitted. In the arbitrary-ideal setting, however, we normally retain
the symbols \(J_\ell^{[0]}\), because the directional module depends
on the entire ordered tuple \((J_1,\ldots,J_s)\), even when the
\(\ell\)-th entry has repetition type zero. Such entries may be omitted
only when the ambient tuple is fixed and no ambiguity can arise. 
The distinguished direction must occur in the type. Thus the
directional coefficient in Definition~\ref{def:directional-mixed-multiplicities}
is defined only when
\(
 d_j>0.
\)

For comparison with multi-index notation, if \(i=|\mathbf d|\), then
\begin{equation}
 \begin{aligned}
 c^{(j)}\bigl(J_1^{[d_1]},\ldots,J_s^{[d_s]};M\bigr)
 &=c_{i,\mathbf d-\mathbf e_j}^{(j)}
 (J_1|\cdots|J_s;M).
 \end{aligned}
 \label{eq:translation-old-directional-notation}
\end{equation}
This identity is only a translation between indexing conventions; the
repetition notation will be used throughout.

\section{Structural properties}
\label{sec:structural-properties}

In this section, we collect the basic  properties of directional mixed
multiplicities.  They parallel the corresponding properties of the
Achilles--Manaresi sequence, but the proofs must keep track of the
distinguished direction.

\begin{lemma}
\label{lem:directional-additivity-exact-sequences}
Let \(0\longrightarrow M_1\longrightarrow M_2\longrightarrow M_3 \longrightarrow0\) be a short exact sequence of nonzero finitely generated \(R\)-modules such that \(\dim_RM_1=\dim_RM_2=\dim_RM_3=D\geq1.\) Suppose that \(J_1,\ldots,J_s\) are proper ideals satisfying \(\operatorname{ht}_{M_q}(J_\ell)>0\) (\(q=1,2,3,\;\ell=1,\ldots,s\)). If \(1\leq|\mathbf d|\leq D\) and \(d_j>0\), then
\begin{align*}
c^{(j)}\bigl(J_1^{[d_1]},\ldots,J_s^{[d_s]};M_2\bigr)
={}&c^{(j)}\bigl(J_1^{[d_1]},\ldots,J_s^{[d_s]};M_1\bigr)+c^{(j)}\bigl(J_1^{[d_1]},\ldots,J_s^{[d_s]};M_3\bigr).
\end{align*}
\end{lemma}

\begin{proof}
First, identify \(M_1\) with its image in \(M_2\), and let \(\pi:M_2\to M_3\) be the given surjection. For \(\mathbf u\in\mathbb N^s\), set \(F_{\mathbf u}:=M_1\cap J^{\mathbf u}M_2\).  Since \(\pi(J^{\mathbf u}M_2)=J^{\mathbf u}M_3\), the map \(\pi\) induces a surjection
\[
\overline\pi_{\mathbf u}: \frac{J^{\mathbf u}M_2}{J^{\mathbf u+\mathbf e_j}M_2} \longrightarrow \frac{J^{\mathbf u}M_3}{J^{\mathbf u+\mathbf e_j}M_3}.
\]
Because \(J^{\mathbf u+\mathbf e_j}M_2\subseteq J^{\mathbf u}M_2\), the modular law gives \(J^{\mathbf u}M_2\cap (M_1+J^{\mathbf u+\mathbf e_j}M_2) =F_{\mathbf u}+J^{\mathbf u+\mathbf e_j}M_2\). It follows that
\[
\ker(\overline\pi_{\mathbf u}) \cong \frac{F_{\mathbf u}} {F_{\mathbf u}\cap J^{\mathbf u+\mathbf e_j}M_2} =\frac{F_{\mathbf u}}{F_{\mathbf u+\mathbf e_j}}.
\]
Therefore, summing the resulting exact sequences over all \(\mathbf u\), we obtain
\begin{equation}
0\longrightarrow \mathscr S_{\mathcal F}^{(j)}(M_1) \longrightarrow \mathscr S^{(j)}(\mathbf J;M_2) \longrightarrow \mathscr S^{(j)}(\mathbf J;M_3) \longrightarrow0,
\label{eq:directional-exact-sequence-proof}
\end{equation}
where \(\mathscr S_{\mathcal F}^{(j)}(M_1) :=\bigoplus_{\mathbf u\in\mathbb N^s} \frac{F_{\mathbf u}}{F_{\mathbf u+\mathbf e_j}} \mathbf T^{\mathbf u}\).  Consider \(\mathcal R_{\mathcal F}(M_1) :=\bigoplus_{\mathbf u\in\mathbb N^s} F_{\mathbf u}\mathbf T^{\mathbf u}\). This is a multigraded submodule of the finite \(\mathcal R(\mathbf J)\)-module \(\mathcal R(\mathbf J;M_2)\). Since \(\mathcal R(\mathbf J)\) is Noetherian, \(\mathcal R_{\mathcal F}(M_1)\) is finitely generated. 

Now, choose homogeneous generators \(x_\nu\mathbf T^{\mathbf a_\nu}\), \(1\leq\nu\leq t\), and choose \(\mathbf c=(c_1,\ldots,c_s)\), with every \(c_\ell>0\), such that \(\mathbf a_\nu\leq\mathbf c\) for all \(\nu\). Finite generation then gives
\begin{equation}
F_{\mathbf c+\mathbf v}=J^{\mathbf v}F_{\mathbf c} \qquad(\mathbf v\in\mathbb N^s).
\label{eq:stable-intersection-filtration}
\end{equation}
Indeed, both sides equal \(\sum_\nu J^{\mathbf c+\mathbf v-\mathbf a_\nu}x_\nu\). Consequently, the tail of \(\mathscr S_{\mathcal F}^{(j)}(M_1)\) is a fixed multigraded translation of \(\mathscr S^{(j)}(\mathbf J;F_{\mathbf c})\).
Moreover, \(J^{\mathbf c}M_1\subseteq F_{\mathbf c}\subseteq M_1\). Thus \(Q:=M_1/F_{\mathbf c}\) is annihilated by \(J^{\mathbf c}\). Since every \(c_\ell>0\) and \(\operatorname{ht}_{M_1}(J_\ell)>0\), we have \(\dim_RQ\leq\dim_R(M_1/J^{\mathbf c}M_1)\leq D-1\). The kernels and cokernels in the comparison between the directional multiform modules of \(F_{\mathbf c}\) and \(M_1\) are supported on \(\operatorname{Supp}_R(Q)\). The multigraded Hilbert--Serre dimension estimate therefore shows that their contributions have total degree at most \(D-2\). Hence fixed translations and this lower-dimensional error do not change the degree-\((D-1)\) homogeneous component. We obtain
\begin{equation}
\bigl[P_{\mathcal F,M_1}^{(j)}\bigr]_{D-1} =\bigl[P_{\mathbf J,M_1}^{(j)}\bigr]_{D-1}.
\label{eq:intersection-filtration-top-part}
\end{equation}

It remains to handle the \(\mathfrak m\)-filtration. Put \(A=\mathscr S_{\mathcal F}^{(j)}(M_1)\), \(B=\mathscr S^{(j)}(\mathbf J;M_2)\), and \(C=\mathscr S^{(j)}(\mathbf J;M_3)\). Give \(B\) its \(\mathfrak m\)-adic filtration, give \(A\) the induced filtration \(E_rA:=A\cap\mathfrak m^rB\), and give \(C=B/A\) the quotient filtration. The latter is exactly the \(\mathfrak m\)-adic filtration, since \(\frac{\mathfrak m^rB+A}{A}=\mathfrak m^rC\). For every \(r\geq0\), there is an exact sequence
\[
0\longrightarrow\frac{E_rA}{E_{r+1}A} \longrightarrow\frac{\mathfrak m^rB}{\mathfrak m^{r+1}B} \longrightarrow\frac{\mathfrak m^rC}{\mathfrak m^{r+1}C} \longrightarrow0.
\]
Indeed, the kernel of the second map is \(\frac{\mathfrak m^rB\cap(\mathfrak m^{r+1}B+A)} {\mathfrak m^{r+1}B} =\frac{\mathfrak m^{r+1}B+E_rA}{\mathfrak m^{r+1}B} \cong\frac{E_rA}{E_{r+1}A}\). Thus passing to associated graded modules is exact here precisely because the filtration on \(A\) is the induced filtration.

Applying Artin--Rees over the relevant Noetherian multigraded algebra, with respect to its degree-zero ideal generated by \(\mathfrak m\), gives an integer \(a\geq0\), independent of the multidegree, such that \(\mathfrak m^rA\subseteq E_rA\subseteq\mathfrak m^{r-a}A\) (\(r\gg0\)). Therefore the cumulative Hilbert function defined by \(E_\bullet A\) is bounded between two fixed translations, in the \(r\)-variable, of the cumulative Hilbert function defined by the intrinsic \(\mathfrak m\)-adic filtration of \(A\). A fixed translation changes only terms of total degree at most \(D-2\). Hence
\begin{equation}
\bigl[P_{\operatorname{gr}_E(A)}\bigr]_{D-1} =\bigl[P_{\operatorname{gr}_{\mathfrak m}(A)}\bigr]_{D-1}.
\label{eq:induced-adic-comparison}
\end{equation}

Additivity of lengths in the exact associated graded sequence, followed by Hilbert--Serre additivity, provides
\[
P_{\mathbf J,M_2,D-1}^{(j)}(r,\mathbf u) =P_{\mathbf J,M_1,D-1}^{(j)}(r,\mathbf u) +P_{\mathbf J,M_3,D-1}^{(j)}(r,\mathbf u),
\]
where we used \eqref{eq:intersection-filtration-top-part} and
\eqref{eq:induced-adic-comparison} (see also \cite[Section~4.1]{BrunsHerzog}). 


Therefore, compare the coefficients of \(r^{D-|\mathbf d|}u_j^{d_j-1} \prod_{\ell\neq j}u_\ell^{d_\ell}\). The three coefficients have the same nonzero normalizing factor
\[
\frac{1}{(D-|\mathbf d|)!\,(d_j-1)! \prod_{\ell\neq j}d_\ell!}.
\]
Cancelling this factor proves the asserted formula.
\end{proof}


The next result shows that, in the one-ideal case, the directional construction recovers the classical multiplicity sequence.

\begin{proposition}
\label{prop:one-ideal-directional}
Let \(I\subsetneq R\) such that \(\operatorname{ht}_M(I)>0\). Then, for  $1\leq i\leq D$, $$c^{(1)}\bigl(I^{[i]};M\bigr)=c_i(I;M).$$
\end{proposition}

\begin{proof}
For \(s=1\),
\[
 \mathscr S^{(1)}(I;M)=\bigoplus_{u\geq0}I^uM/I^{u+1}M
 =\operatorname{gr}_I(M),
\]
so \(\mathscr G^{(1)}(I;M)=\mathscr G(I;M)\). If \(h_{I,M}(r,u)\) is the
doubly cumulative function in
\eqref{eq:multiplicity-sequence-polynomial}, then
\[
 H_{I,M}^{(1)}(r,u)=h_{I,M}(r,u)-h_{I,M}(r,u-1).
\]
Taking the first difference in \(u\) of its degree-\(D\) homogeneous
part gives
\(
 \sum_{i=1}^{D}\frac{c_i(I;M)}{(D-i)!\,(i-1)!}
 r^{D-i}u^{i-1}.
\)
Now, using  \eqref{eq:proved-directional-polynomial-repetition} we have
the assertion. Moreover, \(c_0(I;M)=0\) because \(I\) has positive
height (\cite[Section~2]{PoliniTrungUlrichValidashti}).
\end{proof}

\subsection*{The
\texorpdfstring{\(\mathfrak m\)}{m}-primary case}

For \(\mathfrak m\)-primary ideals, the repetition notation agrees
exactly with the classical mixed-multiplicity notation for \(M\) in
every type of total degree \(D\). In particular, neither
\(\dim R=D\) nor faithfulness of \(M\) is required. The dimension
entering all types and factorials is always \(D=\dim_R M\).

\begin{theorem}
\label{thm:directional-primary-compatibility}
Assume that \(J_1,\ldots,J_s\) are \(\mathfrak m\)-primary. Let
\(
 \mathbf d=(d_1,\ldots,d_s)\in\mathbb N^s,
\)
and choose \(j\) with \(d_j>0\).

\begin{enumerate}[label=\textup{(\arabic*)}]
\item If \(1\leq|\mathbf d|<D\), then $c^{(j)}\bigl(J_1^{[d_1]},\ldots,J_s^{[d_s]};M\bigr)=0.$

\item If \(|\mathbf d|=D\), then
\begin{equation}
 c^{(j)}\bigl(J_1^{[d_1]},\ldots,J_s^{[d_s]};M\bigr)
 =e\bigl(J_1^{[d_1]},\ldots,J_s^{[d_s]};M\bigr).
 \label{eq:directional-equals-classical}
\end{equation}

\item If \(|\mathbf d|=D\), \(d_j>0\), and \(d_\ell>0\), then $c^{(j)}\bigl(J_1^{[d_1]},\ldots,J_s^{[d_s]};M\bigr)
 =c^{(\ell)}\bigl(J_1^{[d_1]},\ldots,J_s^{[d_s]};M\bigr).$
\end{enumerate}
\end{theorem}

\begin{proof}
Fix \(j\), and choose \(a\geq1\) such that
\(\mathfrak m^a\subseteq J_j\). For every \(r\geq a-1\),
\(
 \mathfrak m^{r+1}J^{\mathbf u}M
 \subseteq J_jJ^{\mathbf u}M=J^{\mathbf u+\mathbf e_j}M.
\)
Consequently,
\begin{align*}
 H_{\mathbf J,M}^{(j)}(r,\mathbf u)
 &=\lambda_R\left(\frac{J^{\mathbf u}M}
 {J^{\mathbf u+\mathbf e_j}M}\right)\\
 &=\lambda_R\bigl(M/J^{\mathbf u+\mathbf e_j}M\bigr)
 -\lambda_R\bigl(M/J^{\mathbf u}M\bigr)\\
 &=Q_{\mathbf J,M}(\mathbf u+\mathbf e_j)
 -Q_{\mathbf J,M}(\mathbf u).
\end{align*}
The eventual polynomial is independent of \(r\). In
\eqref{eq:proved-directional-polynomial-repetition}, every type with
\(|\mathbf d|<D\) is multiplied by the positive power
\(r^{D-|\mathbf d|}\); hence all such coefficients vanish.
The degree-\((D-1)\) homogeneous component of the finite difference is
\(\partial Q_{\mathbf J,M,D}/\partial u_j\). By
\eqref{eq:Bhattacharya-repetition}, this derivative is
\[
 \sum_{\substack{d_1+\cdots+d_s=D\\d_j>0}}
 \frac{e\bigl(J_1^{[d_1]},\ldots,J_s^{[d_s]};M\bigr)}
 {(d_j-1)!\prod_{\ell\neq j}d_\ell!}
 u_j^{d_j-1}\prod_{\ell\neq j}u_\ell^{d_\ell}.
\]
Now, using  \eqref{eq:proved-directional-polynomial-repetition} one has
\eqref{eq:directional-equals-classical}. Part~\textup{(3)} follows
because both directional coefficients equal the same classical mixed
multiplicity. 
\end{proof}

\begin{corollary}
\label{cor:primary-pure-types}
If \(J_1,\ldots,J_s\) are \(\mathfrak m\)-primary, then
\[
 c^{(j)}\bigl(
 J_1^{[0]},\ldots,J_j^{[D]},\ldots,J_s^{[0]};M
 \bigr)=e(J_j;M)
 \qquad(1\leq j\leq s).
\]
In particular, when \(s=1\),
\(
 c^{(1)}\bigl(J^{[D]};M\bigr)=c_D(J;M)=e(J;M).
\)
\end{corollary}

\begin{definition}
\label{def:directional-superficial-element}
Let \((R,\mathfrak m,k)\) be a Noetherian local ring, let \(M\) be a
nonzero finitely generated \(R\)-module, and let
\(\mathbf J=(J_1,\ldots,J_s)\) be a tuple of proper ideals.  Fix
\(j,\ell\in\{1,\ldots,s\}\), and put
\(
 A^{(j)}:=\mathscr G^{(j)}(\mathbf J;R),
 \,\,
 E^{(j)}:=\mathscr G^{(j)}(\mathbf J;M).
\)
For \(x\in J_\ell\), denote by \(x^\star\) the homogeneous element of
\(A^{(j)}\) of multidegree \((0,\mathbf e_\ell)\) induced by the class
of \(x\) in \(J_\ell/\mathfrak mJ_\ell\).  Set
\(
 \overline R:=R/xR,
 \,\,
 \overline M:=M/xM,
 \,\,
 \overline J_i:=J_i\overline R.
\)
We say that \(x\) is \emph{
\(j\)-directionally superficial in the \(\ell\)-th direction} for
\(\mathbf J\) with respect to \(M\) if the following conditions hold:

\begin{enumerate}[label=\textup{(\roman*)}]
\item \(x^\star\) is filter-regular on \(E^{(j)}\), that is,
\[
 \bigl(0:_{E^{(j)}}x^\star\bigr)_{q,\mathbf u}=0
 \qquad\text{for every }q\geq0
 \text{ and all }\mathbf u\gg\mathbf0;
\]

\item the canonical comparison homomorphism
\(
 \frac{E^{(j)}}{x^\star E^{(j)}}
 \longrightarrow
 \mathscr G^{(j)}(\overline{\mathbf J};\overline M)
\)
is an isomorphism in degree \((q,\mathbf u)\) for every \(q\geq0\)
and all \(\mathbf u\gg\mathbf0\).
\end{enumerate}

Equivalently, for every \(q\geq0\) and all sufficiently large
\(\mathbf u\), multiplication by \(x^\star\) and passage to \(M/xM\)
give an exact sequence in degree \((q,\mathbf u)\):
\begin{equation}
 0\longrightarrow
 E^{(j)}(-\mathbf e_\ell)
 \xrightarrow{\,x^\star\,}
 E^{(j)}
 \longrightarrow
 \mathscr G^{(j)}(\overline{\mathbf J};\overline M)
 \longrightarrow0,
 \label{eq:directional-superficial-exact-sequence}
\end{equation}
where the shift affects only the \(\mathbf u\)-grading.
This is the directional associated-graded form of the
colon--intersection definition of a superficial element in
\cite[Definition~17.2.1]{HunekeSwanson} (see also
\cite[Proposition~17.2.2 and Lemma~17.2.4]{HunekeSwanson}).
\end{definition}

The preceding definition is designed precisely so that passage to a suitable superficial quotient preserves the expected directional mixed multiplicities. The next theorem gives the corresponding reduction formula.

\begin{theorem}
\label{thm:directional-superficial-reduction}
Let \((R,\mathfrak m,k)\) be a Noetherian local ring, let \(M\) be a
nonzero finitely generated \(R\)-module of dimension \(D>1\), and let
\(J_1,\ldots,J_s\) be arbitrary proper ideals such that
\(
 \operatorname{ht}_M(J_i)>0
 \,\,(i=1,\ldots,s).
\)
Fix \(j,\ell\in\{1,\ldots,s\}\), and let \(x\in J_\ell\) be
\(j\)-directionally superficial in the \(\ell\)-th direction.  Assume
that \(x\) avoids the top-dimensional associated primes of \(M\), and
set
\(
 R':=R/xR,
 \,\,
 M':=M/xM,
 \,\,
 J_i':=J_iR'.
\)
Suppose that
\(
 \dim_{R'}M'=D-1,
 \,\,
 J_i'\subsetneq R',
 \,\,
 \operatorname{ht}_{M'}(J_i')>0
 \,\,(i=1,\ldots,s).
\)
Let \(\mathbf d=(d_1,\ldots,d_s)\in\mathbb N^s\) satisfy
\(
 1\leq |\mathbf d|\leq D,
 \,\,
 d_j>0,
 \,\,
 d_\ell>\delta_{j\ell},
\)
where \(\delta_{j\ell}\) is the Kronecker delta.  Then
\begin{equation}
 c_R^{(j)}\bigl(
 J_1^{[d_1]},\ldots,J_s^{[d_s]};M
 \bigr)
 =
 c_{R'}^{(j)}\bigl(
 (J_1')^{[d_1]},\ldots,
 (J_\ell')^{[d_\ell-1]},\ldots,
 (J_s')^{[d_s]};M'
 \bigr).
\label{eq:directional-superficial-reduction}
\end{equation}
Thus, if \(\ell\neq j\), it is enough that \(d_\ell>0\); if
\(\ell=j\), one must have \(d_j\geq2\), because one occurrence of
\(J_j\) is already used to specify the distinguished direction.
\end{theorem}

\begin{proof}
Write
\(
 P_M^{(j)}(r,\mathbf u)
 :=P_{\mathbf J,M}^{(j)}(r,\mathbf u)
 \,\,\text{and}\,\,
 P_{M'}^{(j)}(r,\mathbf u)
 :=P_{\mathbf J',M'}^{(j)}(r,\mathbf u).
\) Taking the component of multidegree \((q,\mathbf u)\) in
\eqref{eq:directional-superficial-exact-sequence}, and then summing the
lengths for \(0\leq q\leq r\), gives, for all sufficiently large
\(r,\mathbf u\),
\[
 H_{\mathbf J',M'}^{(j)}(r,\mathbf u)
 =H_{\mathbf J,M}^{(j)}(r,\mathbf u)
  -H_{\mathbf J,M}^{(j)}
   (r,\mathbf u-\mathbf e_\ell).
\]
Consequently, the eventual Hilbert polynomials satisfy
\begin{equation}
 P_{M'}^{(j)}(r,\mathbf u)
 =P_M^{(j)}(r,\mathbf u)
  -P_M^{(j)}(r,\mathbf u-\mathbf e_\ell).
 \label{eq:directional-polynomial-difference}
\end{equation}

Since \(\dim_RM=D\), the homogeneous component of degree \(D-1\) of
\(P_M^{(j)}\) is
\begin{align*}
 \sum_{\substack{1\leq|\mathbf a|\leq D\\a_j>0}}
 \frac{c_R^{(j)}\bigl(
 J_1^{[a_1]},\ldots,J_s^{[a_s]};M\bigr)}
 {(D-|\mathbf a|)!\,(a_j-1)!\prod_{i\neq j}a_i!}
 r^{D-|\mathbf a|}u_j^{a_j-1}
 \prod_{i\neq j}u_i^{a_i}.
\end{align*}
Taking the first finite difference in \(u_\ell\) lowers the exponent
of \(u_\ell\) by one.  Its homogeneous component of degree \(D-2\)
is the partial derivative of the displayed degree-\((D-1)\) form with
respect to \(u_\ell\).
 If \(\ell\neq j\), the term of type \(\mathbf d\) contributes
\[
 \frac{c_R^{(j)}\bigl(
 J_1^{[d_1]},\ldots,J_s^{[d_s]};M\bigr)}
 {(D-|\mathbf d|)!\,(d_j-1)!
  (d_\ell-1)!\prod_{i\neq j,\ell}d_i!}
 r^{D-|\mathbf d|}u_j^{d_j-1}u_\ell^{d_\ell-1}
 \prod_{i\neq j,\ell}u_i^{d_i}.
\]
If \(\ell=j\), it contributes
\[
 \frac{c_R^{(j)}\bigl(
 J_1^{[d_1]},\ldots,J_s^{[d_s]};M\bigr)}
 {(D-|\mathbf d|)!\,(d_j-2)!
  \prod_{i\neq j}d_i!}
 r^{D-|\mathbf d|}u_j^{d_j-2}
 \prod_{i\neq j}u_i^{d_i}.
\]
On the other hand, \(\dim_{R'}M'=D-1\), and the coefficient of the
same monomial in the degree-\((D-2)\) homogeneous component of
\(P_{M'}^{(j)}\) is the right-hand side of
\eqref{eq:directional-superficial-reduction} divided by precisely the
same factorials, because
\(
 (D-1)-|\mathbf d-\mathbf e_\ell|=D-|\mathbf d|.
\)
Comparing coefficients in
\eqref{eq:directional-polynomial-difference} proves
\eqref{eq:directional-superficial-reduction}.
\end{proof}

\begin{remark}
\label{rem:directional-superficial-dimension-one}
Suppose that \(D=1\), take \(\ell=j\), and let \(x\in J_j\) be a
superficial parameter for \(J_j\) with respect to \(M\) in the sense
of \cite[Definition~17.2.1]{HunekeSwanson}.  Then the only admissible
pure type in the \(j\)-th direction is
\((0,\ldots,0,1,0,\ldots,0)\), and
\[
 c_R^{(j)}\bigl(
 J_1^{[0]},\ldots,J_j^{[1]},\ldots,J_s^{[0]};M
 \bigr)
 =\lambda_R(M/xM)-\lambda_R(0:_Mx).
\]
Indeed, positive height on the one-dimensional support makes the image
of \(J_j\) \(\mathfrak m\)-primary on \(M\); hence the left-hand side
is \(e(J_j;M)\), and the formula is the dimension-one case of
\cite[Theorem~17.4.6]{HunekeSwanson}.
\end{remark}

The next result gives the directional analogue of the classical associativity formula, reducing the invariant to the top-dimensional minimal components of \(M\).

\begin{theorem}
\label{thm:directional-additivity-reduction-formula}
Let \(M\) be a nonzero finitely generated \(R\)-module of dimension
\(D\geq1\), and let \(J_1,\ldots,J_s\) be proper ideals such that
\(
 \operatorname{ht}_M(J_\ell)>0
 \,\,(\ell=1,\ldots,s).
\)
Set
\(
 \Lambda(M):=\left\{
 \mathfrak p\in\operatorname{Min}(R/\operatorname{Ann}_R(M)):
 \dim(R/\mathfrak p)=D
 \right\}.
\)
Fix \(j\in\{1,\ldots,s\}\), and let
\(\mathbf d=(d_1,\ldots,d_s)\in\mathbb N^s\) satisfy
\(
 1\leq|\mathbf d|\leq D\) and $d_j>0.$
For \(\mathfrak p\in\Lambda(M)\), write
\(
 J_\ell(\mathfrak p)
 :=J_\ell(R/\mathfrak p)
 =\frac{J_\ell+\mathfrak p}{\mathfrak p}.
\)
Then
\begin{align}
 c^{(j)}\bigl(J_1^{[d_1]},\ldots,J_s^{[d_s]};M\bigr)
 ={}&
 \sum_{\mathfrak p\in\Lambda(M)}
 \lambda_{R_{\mathfrak p}}(M_{\mathfrak p})
 \cdot c^{(j)}\bigl(
 J_1(\mathfrak p)^{[d_1]},\ldots,
 J_s(\mathfrak p)^{[d_s]};R/\mathfrak p
 \bigr).
 \label{eq:directional-additivity-reduction-formula}
\end{align}
\end{theorem}

\begin{proof}
We first check that every coefficient on the right-hand side is
defined. Let \(\mathfrak p\in\Lambda(M)\). Such a prime is a
top-dimensional minimal prime of \(\operatorname{Supp}_R(M)\). If
\(J_\ell\subseteq\mathfrak p\), then
\(
 \operatorname{ht}_M(J_\ell)=0,
\)
contrary to the hypothesis. Hence \(J_\ell\not\subseteq\mathfrak p\).
Since \(R/\mathfrak p\) is a local domain, the nonzero proper ideal
\(J_\ell(\mathfrak p)\) has positive height. Thus the directional
multiplicity appearing in
\eqref{eq:directional-additivity-reduction-formula} is well defined. Choose a prime filtration of \(M\),
\begin{equation}
 0=N_0\subset N_1\subset\cdots\subset N_t=M,
 \qquad
 N_a/N_{a-1}\cong R/\mathfrak q_a,
 \label{eq:prime-filtration-M}
\end{equation}
where \(\mathfrak q_a\in\operatorname{Supp}_R(M)\). The existence of
such a filtration is standard for finite modules over Noetherian rings;
(\cite[Proposition~6.4]{AtiyahMacdonald}). Apply the proof of
Lemma~\ref{lem:directional-additivity-exact-sequences} successively to
the short exact sequences
\[
 0\longrightarrow N_{a-1}\longrightarrow N_a
 \longrightarrow R/\mathfrak q_a\longrightarrow0.
\]
The equal-dimension hypothesis in the statement of that lemma is used
only to express all top forms with the same normalization.  Here we
instead take, throughout the filtration, the homogeneous component of
ambient degree \(D-1\).  A factor \(R/\mathfrak q_a\) of dimension
strictly smaller than \(D\) has directional Hilbert polynomial of
degree at most
\(\dim(R/\mathfrak q_a)-1\leq D-2\), and hence contributes nothing.
If some \(J_\ell\subseteq\mathfrak q_a\), its eventual pieces also
vanish in the corresponding positive direction.  Additivity of the
ambient degree-\((D-1)\) forms therefore gives directly
\begin{align}
 c^{(j)}\bigl(J_1^{[d_1]},\ldots,J_s^{[d_s]};M\bigr)
 ={}&\sum_{\substack{1\leq a\leq t\\
                     \dim(R/\mathfrak q_a)=D}} c^{(j)}\bigl(
 J_1(\mathfrak q_a)^{[d_1]},\ldots,
 J_s(\mathfrak q_a)^{[d_s]};R/\mathfrak q_a
 \bigr).
 \label{eq:directional-prime-filtration-sum}
\end{align}

Every \(\mathfrak q_a\in\operatorname{Supp}_R(M)\) with
\(\dim(R/\mathfrak q_a)=D\) belongs to \(\Lambda(M)\). Hence
\eqref{eq:directional-prime-filtration-sum} becomes
\begin{align*}
 c^{(j)}\bigl(J_1^{[d_1]},\ldots,J_s^{[d_s]};M\bigr)
 =\sum_{\mathfrak p\in\Lambda(M)}
 n_{\mathfrak p}
 c^{(j)}\bigl(
 J_1(\mathfrak p)^{[d_1]},\ldots,
 J_s(\mathfrak p)^{[d_s]};R/\mathfrak p
 \bigr),
\end{align*}
where \(n_{\mathfrak p}\) is the number of occurrences of
\(R/\mathfrak p\) among the factors of the prime filtration.
Localizing \eqref{eq:prime-filtration-M} at
\(\mathfrak p\in\Lambda(M)\), a factor
\((R/\mathfrak q_a)_{\mathfrak p}\) can be nonzero only if
\(\mathfrak q_a\subseteq\mathfrak p\). Since
\(\mathfrak p\in\operatorname{Min}(\operatorname{Supp}_R M)\) and
\(\mathfrak q_a\in\operatorname{Supp}_R M\), this forces
\(\mathfrak q_a=\mathfrak p\). Thus all other factors vanish, while
each surviving factor is the residue field
\(\kappa(\mathfrak p)=R_{\mathfrak p}/\mathfrak pR_{\mathfrak p}\).
Additivity of length therefore gives
\(
 n_{\mathfrak p}
 =\lambda_{R_{\mathfrak p}}(M_{\mathfrak p}).
\)
Substituting this equality into the preceding sum proves
\eqref{eq:directional-additivity-reduction-formula}, as desired.
\end{proof}

\subsection*{Relation with the diagonal mixed multiplicities}

The diagonal multiform module used by Callejas--Bedregal and Jorge
P\'erez \cite{CallejasBedregalJorgePerez} is defined as
$$
\mathscr S^\Delta(\mathbf J;M):=\frac{\mathcal R(\mathbf J;M)}{(J_1\cdots J_s)\mathcal R(\mathbf J;M)}=\bigoplus_{\mathbf u\in\mathbb N^s}\frac{J^{\mathbf u}M}{J^{\mathbf u+\mathbf1}M}\mathbf T^{\mathbf u},\qquad \mathbf1=(1,\ldots,1).
$$

Write
\(c_{i,\boldsymbol\alpha}^{\Delta}
(J_1|\cdots|J_s;M)\), with
\(|\boldsymbol\alpha|=i-1\), for its coefficients in codimension
indexing.

\begin{proposition}
\label{prop:diagonal-directional-sum}
For \(1\leq i\leq D\) and
\(|\boldsymbol\alpha|=i-1\),
\begin{equation}
 c_{i,\boldsymbol\alpha}^{\Delta}(J_1|\cdots|J_s;M)
 =\sum_{j=1}^{s}
 c^{(j)}\bigl(
 J_1^{[\alpha_1]},\ldots,J_j^{[\alpha_j+1]},\ldots,
 J_s^{[\alpha_s]};M\bigr).
 \label{eq:diagonal-directional-repetition}
\end{equation}
\end{proposition}

\begin{proof}
For each \(\mathbf u\),  consider the filtration
\[
 J^{\mathbf u}M
 \supset J^{\mathbf u+\mathbf e_1}M
 \supset J^{\mathbf u+\mathbf e_1+\mathbf e_2}M
 \supset\cdots\supset J^{\mathbf u+\mathbf1}M.
\]
The \(j\)-th successive quotient is
\(
 \frac{J^{\mathbf u+\mathbf e_1+\cdots+\mathbf e_{j-1}}M}
 {J^{\mathbf u+\mathbf e_1+\cdots+\mathbf e_j}M},
\)
which is a fixed multigraded translate of the module defining the
\(j\)-directional coefficient.  Additivity of the top homogeneous
parts, together with invariance under fixed translations, gives
\eqref{eq:diagonal-directional-repetition}.  The replacement of the
induced \(\mathfrak m\)-filtrations by their intrinsic
\(\mathfrak m\)-adic filtrations is justified by Artin--Rees, exactly
as in Lemma~\ref{lem:directional-additivity-exact-sequences}.
\end{proof}

In the $\mathfrak m$-primary case, the diagonal coefficients recover the classical mixed multiplicities. In particular, the following corollary recovers \cite[Proposition~5.2 (2b)]{CallejasBedregalJorgePerez}.

\begin{corollary}
\label{cor:diagonal-primary-case}
Assume that \(J_1,\ldots,J_s\) are \(\mathfrak m\)-primary.  If
\(|\boldsymbol\alpha|=D-1\), then
\begin{equation}
 c_{D,\boldsymbol\alpha}^{\Delta}(J_1|\cdots|J_s;M)
 =\sum_{j=1}^{s}
 e\bigl(J_1^{[\alpha_1]},\ldots,
 J_j^{[\alpha_j+1]},\ldots,J_s^{[\alpha_s]};M\bigr).
 \label{eq:diagonal-primary-repetition}
\end{equation}
If \(1\leq i<D\) and \(|\boldsymbol\alpha|=i-1\), then
\(
 c_{i,\boldsymbol\alpha}^{\Delta}(J_1|\cdots|J_s;M)=0.
\)
\end{corollary}

\section{Directional \texorpdfstring{$(FC)$}{(FC)}-sequences and
Rees-type formulas}
\label{sec:directional-FC-Rees}

Keep the notation and hypotheses of
Section~\ref{sec:directional-mixed-multiplicities}.  Put
\(
 I:=J_1\cdots J_s,
 \,\,
 \overline M:=M/(0:_M I^\infty).
\) The positive-height assumption implies that
\(\dim_R\overline M=D\) and
\(\dim_R(0:_M I^\infty)<D\). Fix \(j\).  As in the proof of
Proposition~\ref{prop:directional-polynomiality}, set
\begin{equation}
 E^{(j)}(\mathbf J;M)
 :=\mathscr G^{(j)}(\mathbf J;M)\otimes_k k[Z],
 \qquad
 \deg Z=\mathbf e_0:=(1,\mathbf0).
 \label{eq:cumulative-standard-module}
\end{equation}
Then
\[
 \dim_k E^{(j)}(\mathbf J;M)_{(r,\mathbf u)}
 =H^{(j)}_{\mathbf J,M}(r,\mathbf u).
\]

Let
\(
 \mathbf d=(d_1,\ldots,d_s)\in\mathbb N^s,
 \,\,
 i:=|\mathbf d|,
 \,\,
 1\leq i\leq D,
 \,\,
 d_j>0,
\)
and set
\begin{equation}
 \boldsymbol\kappa^{(j)}(\mathbf d)
 :=(D-i,d_1,\ldots,d_j-1,\ldots,d_s)
 \in\mathbb N^{s+1}.
 \label{eq:directional-multitype}
\end{equation}
Its total degree is $D-1$.  In this notation,
\[
 c_R^{(j)}
 \bigl(J_1^{[d_1]},\ldots,J_s^{[d_s]};M\bigr)
 =e\!\left(
   E^{(j)}(\mathbf J;M);
   \boldsymbol\kappa^{(j)}(\mathbf d)
  \right).
\]

\begin{definition}
\label{def:directional-weak-FC-element}
Let $L$ be a finitely generated multigraded quotient of
$E^{(j)}(\mathbf J;M)$, and let $0\leq\ell\leq s$.  A homogeneous
element $\xi$ of degree $\mathbf e_\ell$ is called a
\emph{directional weak-$(FC)$-element in direction $\ell$ on $L$}
if
\begin{equation}
 (0:_L\xi)_{\mathbf n}=0
 \qquad\text{for every }\mathbf n\gg\mathbf0.
 \label{eq:directional-weak-FC-filter-regular}
\end{equation}
Equivalently, $\xi$ is filter-regular on $L$ with respect to the
multigraded irrelevant ideal.  When $\ell=0$, the element $\xi$ may
be represented by an element of $\mathfrak m$; when
$1\leq\ell\leq s$, it may be represented by an element of
$J_\ell$.
\end{definition}

This allow us to define the following:

\begin{definition}
\label{def:directional-FC-sequence}
A directional weak-$(FC)$-element $\xi$ on $L$ is called a
\emph{directional $(FC)$-element} if, in addition,
\begin{equation}
 \dim\operatorname{Supp}_{++}(L/\xi L)
 =\dim\operatorname{Supp}_{++}(L)-1.
 \label{eq:directional-FC-dimension-drop}
\end{equation}
An ordered homogeneous sequence
\(\boldsymbol\xi=(\xi_1,\ldots,\xi_t)\) is a directional
weak-$(FC)$-sequence, respectively a directional $(FC)$-sequence,
if, for every $q$, the element $\xi_q$ is a directional
weak-$(FC)$-element, respectively a directional $(FC)$-element, on
\[
 E^{(j)}(\mathbf J;M)/
 (\xi_1,\ldots,\xi_{q-1})E^{(j)}(\mathbf J;M).
\]
It has type $\boldsymbol\kappa=(\kappa_0,\ldots,\kappa_s)$ if
exactly $\kappa_\ell$ of its elements have degree
$\mathbf e_\ell$.
\end{definition}


\begin{lemma}
\label{lem:directional-difference-FC}
Let $L$ be as above and let $\xi$ be a directional weak-$(FC)$-element
of degree $\mathbf e_\ell$.  Then, for all $\mathbf n\gg\mathbf0$,
\[
 H_{L/\xi L}(\mathbf n)
 =H_L(\mathbf n)-H_L(\mathbf n-\mathbf e_\ell).
\]
Consequently, the Hilbert polynomials satisfy
\begin{equation}
 P_{L/\xi L}=\Delta_\ell P_L,
 \qquad
 \Delta_\ell P_L(\mathbf n)
 :=P_L(\mathbf n)-P_L(\mathbf n-\mathbf e_\ell).
 \label{eq:directional-FC-polynomial-difference}
\end{equation}
\end{lemma}

\begin{proof}
Note that the multiplication by $\xi$ gives the homogeneous exact sequence
\[
 0\longrightarrow (0:_L\xi)(-\mathbf e_\ell)
 \longrightarrow L(-\mathbf e_\ell)
 \xrightarrow{\,\xi\,}L
 \longrightarrow L/\xi L\longrightarrow0.
\]
By \eqref{eq:directional-weak-FC-filter-regular}, the first term has
zero components in all sufficiently large multidegrees.  Taking
lengths in a sufficiently large multidegree gives the asserted
difference formula, and hence the corresponding polynomial identity.
\end{proof}

\begin{definition}
\label{def:directional-mixed-multiplicity-system-FC}
A directional weak-$(FC)$-sequence
\(
 \boldsymbol\xi=(\xi_1,\ldots,\xi_{D-1})
\)
of type $\boldsymbol\kappa^{(j)}(\mathbf d)$ is called a
\emph{$j$-directional mixed-multiplicity system of type
$\mathbf d$} if
\begin{equation}
 \dim\operatorname{Supp}_{++}
 \left(
  E^{(j)}(\mathbf J;M)/
  \boldsymbol\xi E^{(j)}(\mathbf J;M)
 \right)=0.
 \label{eq:directional-system-zero-dimensional}
\end{equation}
If every member is a directional $(FC)$-element at the corresponding
stage, we call it a \emph{directional $(FC)$-system}.
\end{definition}

The next result characterizes the positive directional mixed
multiplicities in terms of directional \((FC)\)-systems and gives
their length realization.

\begin{theorem}[Existence criterion and coefficient formula]
\label{thm:directional-FC-existence}
Assume that \(k\) is infinite. Then the following conditions are
equivalent:
\begin{enumerate}[label=\textup{(\roman*)}]
\item
$
 c_R^{(j)}
 \bigl(J_1^{[d_1]},\ldots,J_s^{[d_s]};M\bigr)>0;
 $

\item there exists a \(j\)-directional \((FC)\)-system
\(\boldsymbol\xi\) of type \(\mathbf d\).
\end{enumerate}

When these conditions hold, every sufficiently general
\(j\)-directional \((FC)\)-system
\(\boldsymbol\xi\) of type \(\mathbf d\) satisfies
\begin{equation}
 c_R^{(j)}
 \bigl(J_1^{[d_1]},\ldots,J_s^{[d_s]};M\bigr)
 =
 \lambda_k\left(
 \left[
 \frac{E^{(j)}(\mathbf J;M)}
      {\boldsymbol\xi E^{(j)}(\mathbf J;M)}
 \right]_{\mathbf n}
 \right)
 \qquad(\mathbf n\gg\mathbf0).
 \label{eq:directional-FC-length-formula}
\end{equation}
In particular, the right-hand side is eventually constant and is
independent of the sufficiently general system.
\end{theorem}

\begin{proof}
Set \(E:=E^{(j)}(\mathbf J;M)\), and let \(A\) be the standard
multigraded algebra acting on \(E\).

Assume first that
\(
 c_R^{(j)}
 \bigl(J_1^{[d_1]},\ldots,J_s^{[d_s]};M\bigr)>0.
\)
Then the multitype
\(\boldsymbol\kappa^{(j)}(\mathbf d)\) is effective for \(E\).
At each stage, let \(L\) be the quotient of \(E\) by the elements
already chosen. The elements of \(A_{\mathbf e_\ell}\) that fail to
be filter-regular on \(L\) lie in the union
\(
 \bigcup_{\substack{\mathfrak p\in\operatorname{Ass}_A(L)\\
                    \mathfrak p\not\supseteq A_{++}}}
 \mathfrak p\cap A_{\mathbf e_\ell}.
\)

For every direction occurring in the effective multitype, these are
proper \(k\)-linear subspaces of \(A_{\mathbf e_\ell}\). Since \(k\)
is infinite, their complement contains a nonempty Zariski-open set.

The filter-regular positivity criterion for multigraded mixed
multiplicities therefore allows the elements to be chosen
successively and sufficiently generally so that the relevant
dimension drops by one at every step; see
\cite[Theorem~3.4]{ManhViet}. Since
\[
 \bigl|\boldsymbol\kappa^{(j)}(\mathbf d)\bigr|=D-1,
\]
after \(D-1\) steps we obtain a \(j\)-directional \((FC)\)-system
\(\boldsymbol\xi\) of type \(\mathbf d\). This proves
\textup{(i)}\(\Rightarrow\)\textup{(ii)}.

Conversely, suppose that
\(\boldsymbol\xi=(\xi_1,\ldots,\xi_{D-1})\) is a
\(j\)-directional \((FC)\)-system of type \(\mathbf d\).
Applying Lemma~\ref{lem:directional-difference-FC} successively gives
\begin{equation}
 P_{E/\boldsymbol\xi E}
 =
 \Delta_0^{D-i}
 \Delta_1^{d_1}\cdots
 \Delta_j^{d_j-1}\cdots
 \Delta_s^{d_s}P_E.
 \label{eq:iterated-directional-difference}
\end{equation}
The total order of the difference operator on the right-hand side is
\(D-1\). By the factorial normalization of the top homogeneous part
of \(P_E\), its value is the constant
\(
 c_R^{(j)}
 \bigl(J_1^{[d_1]},\ldots,J_s^{[d_s]};M\bigr).
\)

By Definition~\ref{def:directional-mixed-multiplicity-system-FC},
the quotient \(E/\boldsymbol\xi E\) has nonempty relevant support of
dimension zero. Its Hilbert polynomial is therefore a positive
constant, and its Hilbert function agrees with that constant in all
sufficiently large multidegrees. Hence
\(
 c_R^{(j)}
 \bigl(J_1^{[d_1]},\ldots,J_s^{[d_s]};M\bigr)>0,
\)
which proves \textup{(ii)}\(\Rightarrow\)\textup{(i)} and, at the
same time, establishes
\eqref{eq:directional-FC-length-formula}.

Finally, the filter-regular and dimension-drop conditions are open
conditions in the corresponding homogeneous components. Thus the
same argument applies to every sufficiently general system of the
prescribed type. Since all such systems compute the same coefficient,
the eventual length in
\eqref{eq:directional-FC-length-formula} is independent of the
chosen sufficiently general system.
\end{proof}

\begin{remark}
\label{rem:positivity-effective-type}
The assumption that \(J_1,\ldots,J_s\) are proper ideals of positive
height does not imply that every directional mixed multiplicity is
positive. It guarantees the dimension bound needed to define the
directional coefficients, but positivity also depends on the
effectiveness of the prescribed multitype
\(\boldsymbol\kappa^{(j)}(\mathbf d)\).

For example, let \(R=k[[x,y]]\), \(M=R\), and \(D=2\). If
\(J_1=\mathfrak m\), then
\[
 c_R^{(1)}(J_1^{[1]};R)=c_1(J_1;R)=0,
\]
although \(J_1\) is proper and \(\mathfrak m\)-primary. Similarly,
if \(J_1=(x)\), then \(\operatorname{ht}(J_1)=\ell(J_1)=1\), and
\[
 c_R^{(1)}(J_1^{[2]};R)=c_2(J_1;R)=0.
\]
Thus properness and positive height do not guarantee positivity,
even for a type of total degree \(D\).

Accordingly, Theorem~\ref{thm:directional-FC-existence} shows that
positivity is not merely a technical hypothesis: it is precisely the
condition guaranteeing the existence of a directional
\((FC)\)-system of the prescribed type. When the coefficient
vanishes, one may still choose filter-regular or weak-\((FC)\)
elements in the active directions, but the required dimension drops
need not occur, and hence the resulting sequence need not be a
directional \((FC)\)-system.
\end{remark}

\begin{remark}
\label{rem:positivity-effective-type}
Theorem~\ref{thm:directional-FC-existence} is intentionally stated
for positive coefficients.  An infinite residue field guarantees a
generic filter-regular element in every active direction, but it does
not make every prescribed mixed type effective.  Thus infinitude of
$k$ does not by itself guarantee a directional $(FC)$-system of an
arbitrary type.  This is precisely the defect in definitions that
require every $J_\ell$ to retain positive height after every quotient.
\end{remark}

We next isolate the additional condition needed to turn the intrinsic
multigraded formula into a Hilbert--Samuel multiplicity on $M$.

\begin{definition}
\label{def:synchronized-Rees-realization}
Let $\boldsymbol\xi$ be a $j$-directional $(FC)$-system of type
$\mathbf d$.  Choose lifts of its elements and order them as
\(
 a_1,\ldots,a_{D-i}\in\mathfrak m,
 \,\,
 x_t\in J_{\ell_t}\,\,(1\leq t<i),
\)
where the multiset $\ell_1,\ldots,\ell_{i-1}$ contains $\ell\ne j$
exactly $d_\ell$ times and contains $j$ exactly $d_j-1$ times.

We say that $\boldsymbol\xi$ admits a
\emph{synchronized Rees realization in direction $j$} if there is
$z\in J_j$ such that, with
\(
 L=(a_1,\ldots,a_{D-i},x_1,\ldots,x_{i-1}),
 \,\,
 N=\overline M/L\overline M,
 \,\,
 Q=(L,z),
\)
the following conditions hold:
\begin{enumerate}[label=\textup{(R\arabic*)}]
\item $a_1,\ldots,a_{D-i},x_1,\ldots,x_{i-1}$ is a
      $Q$-superficial sequence on $\overline M$,
      $\dim_RN=1$, and $\bar z$ is a parameter on $N$;
\item $Q$ is a parameter ideal on $M$ (and hence on $\overline M$),
      and the lifted
      family is a joint reduction of
      $(\mathfrak m,J_1,\ldots,J_s)$ of type
      $(D-i,d_1,\ldots,d_s)$, that is,
\begin{align}
 \mathfrak m^{n_0}J^{\mathbf u}\overline M
 &=\sum_{h=1}^{D-i}
   a_h\mathfrak m^{n_0-1}J^{\mathbf u}\overline M
 +
   \sum_{t=1}^{i-1}
   x_t\mathfrak m^{n_0}
   J^{\mathbf u-\mathbf e_{\ell_t}}\overline M
   +z\mathfrak m^{n_0}
   J^{\mathbf u-\mathbf e_j}\overline M
 \label{eq:synchronized-joint-reduction}
\end{align}
for all $n_0\gg0$ and $\mathbf u\gg\mathbf0$;
\item the final multigraded quotient is synchronized with the
      one-ideal directional module of $(\bar z)$ on $N$, in the
      following numerical sense: for all sufficiently large
      $r,n$ and $\mathbf u$,
\begin{equation}
 \lambda_k\!\left(
 \left[
  \frac{E^{(j)}(\mathbf J;M)}
       {\boldsymbol\xi E^{(j)}(\mathbf J;M)}
 \right]_{(r,\mathbf u)}
 \right)
 =
 \lambda_R\!\left(
 \frac{\bar z^{\,n}N}
      {\mathfrak m^{r+1}\bar z^{\,n}N+
       \bar z^{\,n+1}N}
 \right).
 \label{eq:terminal-synchronization}
\end{equation}
\end{enumerate}
\end{definition}

The synchronization condition provides the precise bridge between the directional coefficient and the classical one-dimensional multiplicity theory. The next result shows that, under this compatibility, the directional mixed multiplicity is realized by the multiplicity of the corresponding joint reduction.

\begin{theorem}
\label{thm:directional-Rees-FC}
Assume that a $j$-directional $(FC)$-system of type $\mathbf d$
admits a synchronized Rees realization.  With the notation of
Definition~\ref{def:synchronized-Rees-realization}, one has
\begin{equation}
 c_R^{(j)}
 \bigl(J_1^{[d_1]},\ldots,J_s^{[d_s]};M\bigr)
 =c_1((\bar z);N)
 =e_R((\bar z);N)
 =e_R(Q;\overline M).
 \label{eq:directional-Rees-FC-formula}
\end{equation}
Moreover, modulo $L\overline M$, the joint-reduction equation gives
\begin{equation}
 \mathfrak m^{n_0}J^{\mathbf u}N
 =\bar z\mathfrak m^{n_0}
  J^{\mathbf u-\mathbf e_j}N
 \qquad(n_0\gg0,\ \mathbf u\gg\mathbf0).
 \label{eq:directional-terminal-joint-reduction}
\end{equation}
Therefore,
\begin{equation}
 c_R^{(j)}
 \bigl(J_1^{[d_1]},\ldots,J_s^{[d_s]};M\bigr)
 =e_R(Q;M).
 \label{eq:directional-Rees-on-M}
\end{equation}
\end{theorem}
\begin{proof}
Set $E:=E^{(j)}(\mathbf J;M)$. Since $\boldsymbol\xi$ is a $j$-directional $(FC)$-system of type $\mathbf d$, Theorem~\ref{thm:directional-FC-existence} gives
\begin{equation}
c_R^{(j)}\bigl(J_1^{[d_1]},\ldots,J_s^{[d_s]};M\bigr) = \lambda_k\left(\left[ \frac{E}{\boldsymbol\xi E} \right]_{(r,\mathbf u)}\right) \qquad(r,\mathbf u\gg0).
\label{eq:directional-Rees-terminal-length}
\end{equation}

By the synchronization condition \textup{(R3)}, the right-hand side of \eqref{eq:directional-Rees-terminal-length} agrees, for all sufficiently large $r,n$, and $\mathbf u$, with
\begin{equation}
\lambda_R\left( \frac{\bar z^{\,n}N}{\mathfrak m^{r+1}\bar z^{\,n}N+ \bar z^{\,n+1}N} \right).
\label{eq:directional-Rees-one-ideal-function}
\end{equation}
The expression in \eqref{eq:directional-Rees-one-ideal-function} is precisely the cumulative directional Hilbert function of the one-ideal filtration defined by $(\bar z)$ on $N$; namely, $H_{(\bar z),N}^{(1)}(r,n) = \lambda_R\left( \frac{\bar z^{\,n}N}{\mathfrak m^{r+1}\bar z^{\,n}N+ \bar z^{\,n+1}N} \right)$. Consequently,
\begin{equation}
c_R^{(j)}\bigl(J_1^{[d_1]},\ldots,J_s^{[d_s]};M\bigr) = H_{(\bar z),N}^{(1)}(r,n) \qquad(r,n\gg0).
\label{eq:directional-equals-terminal-Hilbert-function}
\end{equation}

By condition \textup{(R1)}, one has $\dim_RN=1$, and $\bar z$ is a parameter on $N$. Hence $(\bar z)$ is an ideal of definition for $N$. The one-ideal compatibility and the $\mathfrak m$-primary case therefore give $H_{(\bar z),N}^{(1)}(r,n) = c_1((\bar z);N) = e_R((\bar z);N)$ for $r,n\gg0$. Combining this equality with \eqref{eq:directional-equals-terminal-Hilbert-function}, we obtain
\begin{equation}
c_R^{(j)}\bigl(J_1^{[d_1]},\ldots,J_s^{[d_s]};M\bigr) = c_1((\bar z);N) = e_R((\bar z);N).
\label{eq:directional-terminal-multiplicity}
\end{equation}

We next compare the terminal multiplicity with the multiplicity of $Q$ on $\overline M$. Write $y_1,\ldots,y_{D-1} := a_1,\ldots,a_{D-i},x_1,\ldots,x_{i-1}$. By condition \textup{(R1)}, the sequence $y_1,\ldots,y_{D-1}$ is $Q$-superficial on $\overline M$, and $N = \overline M/(y_1,\ldots,y_{D-1})\overline M$ has dimension one. Repeated application of the superficial reduction formula for Hilbert--Samuel multiplicity gives
\begin{equation}
e_R(Q;\overline M)=e_R(QN;N).
\label{eq:superficial-reduction-to-N}
\end{equation}
Since the elements of $L$ vanish on $N$, we have $QN=(L,z)N=\bar zN$. Therefore,
\begin{equation}
e_R(Q;\overline M) = e_R(QN;N) = e_R((\bar z);N).
\label{eq:terminal-equals-Q-multiplicity}
\end{equation}
Equations \eqref{eq:directional-terminal-multiplicity} and \eqref{eq:terminal-equals-Q-multiplicity} prove $c_R^{(j)}\bigl(J_1^{[d_1]},\ldots,J_s^{[d_s]};M\bigr) = c_1((\bar z);N) = e_R((\bar z);N) = e_R(Q;\overline M)$, which is \eqref{eq:directional-Rees-FC-formula}.

We now prove the terminal joint-reduction equality. Reducing \eqref{eq:synchronized-joint-reduction} modulo $L\overline M$, all the summands involving $a_h$ or $x_t$ vanish, because $a_h,x_t\in L$. Thus only the term involving $z$ remains, and we obtain
\[
\mathfrak m^{n_0}J^{\mathbf u}N = \bar z\,\mathfrak m^{n_0} J^{\mathbf u-\mathbf e_j}N \qquad(n_0\gg0,\ \mathbf u\gg\mathbf0).
\]
This proves \eqref{eq:directional-terminal-joint-reduction}.

Finally, put $T:=0:_M I^\infty$. By the positive-height hypothesis, $\dim_RT<D$ and $\dim_RM=\dim_R\overline M=D$. Condition \textup{(R2)} says that $Q$ is a parameter ideal on $M$. In particular, $Q$ is an ideal of definition for $M$ and for $\overline M$. Applying the additivity theorem for Hilbert--Samuel multiplicity to the exact sequence $0\longrightarrow T \longrightarrow M \longrightarrow\overline M \longrightarrow0$, the module $T$ contributes nothing to the degree-$D$ part of the Hilbert--Samuel polynomial, since $\dim_RT<D$. Hence
\begin{equation}
e_R(Q;M)=e_R(Q;\overline M).
\label{eq:multiplicity-removes-I-torsion}
\end{equation}
Combining \eqref{eq:multiplicity-removes-I-torsion} with \eqref{eq:directional-Rees-FC-formula} gives $c_R^{(j)}\bigl(J_1^{[d_1]},\ldots,J_s^{[d_s]};M\bigr) = e_R(Q;M)$, which proves \eqref{eq:directional-Rees-on-M}.
\end{proof}

\begin{corollary}
\label{cor:top-directional-Rees-FC}
Suppose that $|\mathbf d|=D$ and that a $j$-directional $(FC)$-system
\(
 \boldsymbol\xi=(\xi_1,\ldots,\xi_{D-1})
\)
admits a synchronized Rees realization with lifts
$x_t\in J_{\ell_t}$ and terminal element $z\in J_j$.  Then, for
$Q=(x_1,\ldots,x_{D-1},z)$,
\[
 c_R^{(j)}
 \bigl(J_1^{[d_1]},\ldots,J_s^{[d_s]};M\bigr)
 =e_R(Q;M),
\]
and
\[
 J^{\mathbf u}\overline M
 =\sum_{t=1}^{D-1}
   x_tJ^{\mathbf u-\mathbf e_{\ell_t}}\overline M
  +zJ^{\mathbf u-\mathbf e_j}\overline M
 \qquad(\mathbf u\gg\mathbf0).
\]
\end{corollary}

In the \(\mathfrak m\)-primary case, the directional Rees-type formula specializes to the classical joint-reduction formula for mixed multiplicities. Thus the preceding theorem recovers Rees' theorem within the directional framework.

\begin{corollary}
\label{cor:classical-Rees-from-directional-FC}
Assume that $k$ is infinite and that $J_1,\ldots,J_s$ are
$\mathfrak m$-primary.  Let
\(
 d_1+\cdots+d_s=D.
\)
Then there are elements
\(
 x_{\ell,1},\ldots,x_{\ell,d_\ell}\in J_\ell
 \,\,(1\leq\ell\leq s)
\)
forming a superficial joint reduction of type
$(d_1,\ldots,d_s)$ with respect to $M$.  If
\(
 Q=(x_{\ell,t}\mid
     1\leq\ell\leq s,\ 1\leq t\leq d_\ell),
\)
then $Q$ is a parameter ideal on $M$ and
\begin{equation}
 e_R\bigl(J_1^{[d_1]},\ldots,J_s^{[d_s]};M\bigr)
 =e_R(Q;M).
 \label{eq:classical-Rees-corollary-FC}
\end{equation}
More precisely, if $j$ is any index with $d_j>0$ and one of the
$J_j$-elements is designated as the terminal element $z$, then the
remaining $D-1$ elements determine a synchronized
$j$-directional Rees realization and
\begin{equation}
 e_R\bigl(J_1^{[d_1]},\ldots,J_s^{[d_s]};M\bigr)
 =c_R^{(j)}
  \bigl(J_1^{[d_1]},\ldots,J_s^{[d_s]};M\bigr)
 =e_R(Q;M).
 \label{eq:classical-directional-Rees-corollary}
\end{equation}
\end{corollary}

\begin{proof}
Over an infinite residue field, the elements may be chosen
successively in the intersection of the nonempty open sets defining
classical superficiality, directional filter-regularity, and the
joint-reduction condition (\cite[Proposition~17.2.2]{HunekeSwanson} and
\cite[Lemma~2.2 and Proposition~2.3]{VietDinhThanh}).  The resulting
sequence contains exactly \(d_\ell\) elements from \(J_\ell\),
generates a parameter ideal, and is a joint reduction of the
prescribed type.  The successive associated-graded comparison maps
give condition \textup{(R3)}, so the sequence is a synchronized
directional Rees realization.
Since the $J_\ell$ are $\mathfrak m$-primary,
\(
 0:_M I^\infty=H^0_{\mathfrak m}(M),
\)
which has dimension smaller than $D$.  Termwise compatibility of the
directional coefficients with Rees--Teissier mixed multiplicities
gives, for every $j$ with $d_j>0$,
\[
 c_R^{(j)}
 \bigl(J_1^{[d_1]},\ldots,J_s^{[d_s]};M\bigr)
 =e_R\bigl(J_1^{[d_1]},\ldots,J_s^{[d_s]};M\bigr).
\]
Theorem~\ref{thm:directional-Rees-FC} now yields
\eqref{eq:classical-Rees-corollary-FC} and
\eqref{eq:classical-directional-Rees-corollary}.
\end{proof}

\begin{remark}
\label{rem:comparison-Viet-Thanh}
The hierarchy used here follows the theory initiated by D.~Q.~Vi\^et. In \cite{Viet2000}, weak-$(FC)$-elements satisfy a Rees intersection condition and $I$-filter-regularity, while $(FC)$-elements additionally satisfy the expected dimension drop. Over an infinite residue field, suitable weak-$(FC)$-elements form a nonempty Zariski-open set, and Rees superficial sequences yield joint reductions (\cite[Propositions~2.3 and~2.6]{VietDinhThanh}). The corresponding Rees-type theorem for ordinary mixed multiplicities requires a height condition and a joint reduction that is a system of parameters (\cite[Theorem~3.1 and Remark~3.5]{VietDinhThanh}). Note that broader criteria were later obtained in \cite{ThanhViet}. In our setting, however, the quotient $J^{\mathbf u}M/J^{\mathbf u+\mathbf e_j}M$ retains the distinguished direction $j$, so the classical theory does not transfer directly. This motivates defining directional weak-$(FC)$-elements on $E^{(j)}(\mathbf J;M)$ and introducing the synchronization condition needed to recover a Hilbert--Samuel multiplicity. This condition is automatic in the classical $\mathfrak m$-primary case, but not for arbitrary ideals.

\end{remark}

\begin{remark}
The infinitude of $k$ is used only for generic choices. If $k$ is finite, one may pass to the faithfully flat local extension
$
R'=R[X]_{\mathfrak m R[X]},
$
which preserves the relevant numerical invariants and has infinite residue field. However, a joint reduction over $R'$ need not descend to $R$; hence existence in $R$ is asserted only when $k$ is infinite.
\end{remark}

 \begin{remark}
\label{rem:ordinary-joint-reduction-not-enough}
For arbitrary ideals, a joint reduction in the usual sense is not
sufficient to identify a directional mixed multiplicity with the
multiplicity of the ideal generated by the reduction.  Indeed, let $R=k\mathopen{[[}x,y\mathclose{]]}$, $M=R$, $J_1=(x)$, and $J_2=(y).$
The elements \(x\in J_1\) and \(y\in J_2\) form a joint reduction of
type \((1,1)\), since, for every \(n_1,n_2\geq0\),
\begin{align*}
 J_1^{n_1+1}J_2^{n_2+1}
 ={}&xJ_1^{n_1}J_2^{n_2+1}
    +yJ_1^{n_1+1}J_2^{n_2}.
\end{align*}
However, in direction \(j=1\),
\begin{align*}
 H^{(1)}(r,u_1,u_2)
 &=\lambda_R\!\left(
 \frac{x^{u_1}y^{u_2}R}
 {\mathfrak m^{r+1}x^{u_1}y^{u_2}R+x^{u_1+1}y^{u_2}R}
 \right)\\
 &=\lambda_R\!\left(\frac{R}{(\mathfrak m^{r+1},x)}\right)
 =r+1.
\end{align*}
This polynomial contains no \(u_2\)-term, hence $c_R^{(1)}\bigl(J_1^{[1]},J_2^{[1]};R\bigr)=0.$
On the other hand, the joint reduction generates
\(Q=(x,y)=\mathfrak m\), and therefore $c_2(Q;R)=e(Q;R)=1.$
Consequently,
\[
 c_R^{(1)}\bigl(J_1^{[1]},J_2^{[1]};R\bigr)
 \neq c_2((x,y);R).
\]
Thus a Rees-type formula for arbitrary ideals requires an additional
condition controlling the directional grading.  The ordinary
joint-reduction equation controls the classical multigraded product,
but does not separate the coefficients belonging to the different
directions.
\end{remark}

\begin{remark}
\label{rem:directional-rees-single-ideal}
When \(s=1\), there is no interaction among directions and, by
definition, $c_R^{(1)}\bigl(J_1^{[i]};M\bigr)=c_i(J_1;M).$
Thus the obstruction in
Remark~\ref{rem:ordinary-joint-reduction-not-enough} arises only when
one tries to replace a family of ideals by the single ideal generated
by a joint reduction.
\end{remark}

\begin{remark}
For arbitrary positive-height ideals, if \(d_j,d_\ell>0\), there is no
a priori reason for
\[
 c^{(j)}\bigl(J_1^{[d_1]},\ldots,J_s^{[d_s]};M\bigr)
 =c^{(\ell)}\bigl(J_1^{[d_1]},\ldots,J_s^{[d_s]};M\bigr).
\]
The superscript \((j)\) therefore contains genuine information. The
preceding theorem shows that this dependence disappears for
\(\mathfrak m\)-primary ideals in total type \(D\).
\end{remark}

\section{Product formulas}
\label{sec:product-formulas}

The following formula expresses the multiplicity-sequence components of
\(J_1^{n_1}\cdots J_s^{n_s}\) in terms of the directional mixed
multiplicities of \(J_1,\ldots,J_s\). In the \(\mathfrak m\)-primary
case, it specializes to the classical Bhattacharya--Rees--Teissier formula.

\begin{theorem}
\label{thm:directional-product-formula-powers}
Let \(M\) be a nonzero finite \(R\)-module of dimension \(D\geq1\),
and let \(J_1,\ldots,J_s\) be proper ideals satisfying
\(\operatorname{ht}_M(J_\ell)>0\) for every \(\ell\).  Let
\(\mathbf n=(n_1,\ldots,n_s)\in\mathbb N_{>0}^{s}\), and set
\(I(\mathbf n):=J_1^{n_1}\cdots J_s^{n_s}\).  Then, for every
\(1\leq i\leq D\),
\begin{equation}
\begin{aligned}
c_i\bigl(I(\mathbf n);M\bigr)
={}&\sum_{d_1+\cdots+d_s=i}
 \ \sum_{\substack{1\leq j\leq s\\d_j>0}}
 \frac{(i-1)!}{d_1!\cdots(d_j-1)!\cdots d_s!}\cdot c^{(j)}\bigl(J_1^{[d_1]},\ldots,J_s^{[d_s]};M\bigr)
 n_1^{d_1}\cdots n_s^{d_s}.
\end{aligned}
\label{eq:directional-product-formula-powers}
\end{equation}
Equivalently, in the coefficient notation \(c_{i,\boldsymbol\alpha}^{(j)}\), one has
\begin{equation}
c_i\bigl(I(\mathbf n);M\bigr)
=\sum_{|\boldsymbol\alpha|=i-1}
 \binom{i-1}{\boldsymbol\alpha}\mathbf n^{\boldsymbol\alpha}
 \sum_{j=1}^{s}n_jc_{i,\boldsymbol\alpha}^{(j)}
 (J_1|\cdots|J_s;M).
\label{eq:directional-product-multiindex}
\end{equation}
\end{theorem}

\begin{proof}
Write \(J^{\mathbf u}:=J_1^{u_1}\cdots J_s^{u_s}\) and put \(N=n_1+\cdots+n_s\). Choose a lattice path \(\mathbf0=\mathbf b_0,\mathbf b_1,\ldots,\mathbf b_N=\mathbf n\) such that \(\mathbf b_{t+1}-\mathbf b_t=\mathbf e_{j_t}\) for some \(j_t\in\{1,\ldots,s\}\), and such that the direction \(j\) occurs exactly \(n_j\) times. For every \(q\geq0\), this path gives a filtration
\[
J^{q\mathbf n}M
=J^{q\mathbf n+\mathbf b_0}M
\supseteq J^{q\mathbf n+\mathbf b_1}M
\supseteq\cdots\supseteq
J^{q\mathbf n+\mathbf b_N}M
=J^{(q+1)\mathbf n}M.
\]
The \(t\)-th successive quotient is
\[
\frac{J^{q\mathbf n+\mathbf b_t}M}
 {J^{q\mathbf n+\mathbf b_{t+1}}M}
=\frac{J^{q\mathbf n+\mathbf b_t}M}
 {J^{q\mathbf n+\mathbf b_t+\mathbf e_{j_t}}M},
\]
which is a homogeneous component of the \(j_t\)-directional multiform module.
Equip \(J^{q\mathbf n}M/J^{(q+1)\mathbf n}M\) with its
\(\mathfrak m\)-adic filtration and the successive quotients above
with the induced filtrations.  Additivity of lengths gives an
equality of their cumulative Hilbert functions.
The induced filtrations on the successive quotients need not coincide
exactly with their intrinsic \(\mathfrak m\)-adic filtrations.
Nevertheless, Artin--Rees shows that they are good filtrations.
Replacing an induced good filtration by the intrinsic
\(\mathfrak m\)-adic filtration does not change the homogeneous
component of total degree \(D-1\).  Consequently,
\begin{equation}
\bigl[P_{I(\mathbf n),M}^{(1)}(r,q)\bigr]_{D-1}
=\sum_{t=0}^{N-1}
 \bigl[P_{\mathbf J,M}^{(j_t)}
 (r,q\mathbf n+\mathbf b_t)\bigr]_{D-1}.
\label{eq:path-top-parts}
\end{equation}

Each \(\mathbf b_t\) is fixed.  Hence translation by
\(\mathbf b_t\) changes only terms of total degree at most \(D-2\).
Since direction \(j\) occurs \(n_j\) times along the path,
\eqref{eq:path-top-parts} becomes
\begin{equation}
\bigl[P_{I(\mathbf n),M}^{(1)}(r,q)\bigr]_{D-1}
=\sum_{j=1}^{s}n_jP_{\mathbf J,M,D-1}^{(j)}
(r,q\mathbf n).
\label{eq:weighted-directional-polynomial}
\end{equation}

By the one-ideal recovery property, the left-hand side is
\[
\sum_{i=1}^{D}
\frac{c_i(I(\mathbf n);M)}{(D-i)!\,(i-1)!}
r^{D-i}q^{i-1}.
\]
On the other hand, the directional expansion gives
\begin{align*}
P_{\mathbf J,M,D-1}^{(j)}(r,q\mathbf n)
={}&\sum_{\substack{\mathbf d\in\mathbb N^s\\
                     1\leq|\mathbf d|\leq D,\ d_j>0}}
\frac{c^{(j)}
 \bigl(J_1^{[d_1]},\ldots,J_s^{[d_s]};M\bigr)}
 {(D-|\mathbf d|)!\,(d_j-1)!
  \prod_{\ell\neq j}d_\ell!}\cdot r^{D-|\mathbf d|}q^{|\mathbf d|-1}
n_j^{d_j-1}\prod_{\ell\neq j}n_\ell^{d_\ell}.
\end{align*}
Multiplication by the external factor \(n_j\) in
\eqref{eq:weighted-directional-polynomial} changes the last factor
into \(\mathbf n^{\mathbf d}\).  Comparing the coefficients of
\(r^{D-i}q^{i-1}\) and multiplying by
\((D-i)!(i-1)!\) gives
\eqref{eq:directional-product-formula-powers}.  Formula
\eqref{eq:directional-product-multiindex} follows by putting
\(\boldsymbol\alpha=\mathbf d-\mathbf e_j\).
\end{proof}

In the \(\mathfrak m\)-primary case, the preceding product formula recovers the classical Bhattacharya--Rees--Teissier formula (see \cite[Theorem~17.4.2 and Definition~17.4.3]{HunekeSwanson}).

\begin{corollary}
\label{cor:classical-product-formula-from-directional}
Let \(M\) be a nonzero finite \(R\)-module of dimension \(D\geq1\),
and let \(J_1,\ldots,J_s\) be \(\mathfrak m\)-primary ideals.  Then,
for every \(\mathbf n=(n_1,\ldots,n_s)\in\mathbb N^s\), one has
\begin{equation}
\begin{aligned}
e\bigl(J_1^{n_1}\cdots J_s^{n_s};M\bigr)
={}&\sum_{\substack{\mathbf d\in\mathbb N^s\\|\mathbf d|=D}}
\frac{D!}{d_1!\cdots d_s!}\cdot
e\bigl(J_1^{[d_1]},\ldots,J_s^{[d_s]};M\bigr)
n_1^{d_1}\cdots n_s^{d_s}.
\end{aligned}
\label{eq:classical-product-formula-from-directional}
\end{equation}
\end{corollary}

\begin{proof}
First suppose that \(n_1,\ldots,n_s>0\).  The product
\(I(\mathbf n)\) is \(\mathfrak m\)-primary.  Consequently,
\(c_i(I(\mathbf n);M)=0\) for \(1\leq i<D\), and
\(c_D(I(\mathbf n);M)=e(I(\mathbf n);M)\).  Moreover,
\begin{equation}
c^{(j)}
\bigl(J_1^{[d_1]},\ldots,J_s^{[d_s]};M\bigr)
=e\bigl(J_1^{[d_1]},\ldots,J_s^{[d_s]};M\bigr)
\label{eq:directional-classical-primary-corollary}
\end{equation}
whenever \(d_1+\cdots+d_s=D\) and \(d_j>0\), by
Theorem~\ref{thm:directional-primary-compatibility}. Applying Theorem~\ref{thm:directional-product-formula-powers} with
\(i=D\), we obtain
\begin{align*}
e(I(\mathbf n);M)
={}&\sum_{d_1+\cdots+d_s=D}\sum_{j:d_j>0}
\binom{D-1}{d_1,\ldots,d_j-1,\ldots,d_s}\cdot
e\bigl(J_1^{[d_1]},\ldots,J_s^{[d_s]};M\bigr)
\mathbf n^{\mathbf d}.
\end{align*}
For every fixed \(\mathbf d\) with \(|\mathbf d|=D\), one has

\begin{align*}
 \sum_{j:d_j>0}
 \binom{D-1}{d_1,\ldots,d_j-1,\ldots,d_s}
 &=
 \sum_{j:d_j>0}
 \frac{(D-1)!}
 {d_1!\cdots(d_j-1)!\cdots d_s!} \\
 &=
 \frac{(D-1)!}{d_1!\cdots d_s!}
 \sum_{j=1}^{s}d_j \\
 &=
 \frac{D!}{d_1!\cdots d_s!}.
\end{align*}

This gives \eqref{eq:classical-product-formula-from-directional}.
If some \(n_\ell=0\), omit the corresponding ideals and apply the
formula to the subtuple \((J_\ell\mid n_\ell>0)\).  Reintroducing the
omitted indices yields the same identity, since precisely the terms
with \(d_\ell>0\) vanish.
\end{proof}

\begin{remark}
\label{rem:comparison-CBJP-product-formula}
It should be noted that Theorem~\ref{thm:directional-product-formula-powers} refines and extends the product formula of Callejas--Bedregal and Jorge P\'erez \cite[Theorem~5.3]{CallejasBedregalJorgePerez}. In fact, the original Achilles--Manaresi indexing, their formula reads
\begin{equation}
c_k^{\mathrm{AM}}(J_1\cdots J_s;R) = \sum_{|\boldsymbol\alpha|=D-1-k} \binom{D-1-k}{\boldsymbol\alpha} c_{k,\boldsymbol\alpha}^{\mathrm{CBJP}} (J_1|\cdots|J_s;R),
\label{eq:CBJP-product-original-indexing}
\end{equation}
where \(\binom{D-1-k}{\boldsymbol\alpha} = \frac{(D-1-k)!}{\alpha_1!\cdots\alpha_s!}\). Under the codimension translation \(i=D-k\), \(c_i=c_{D-i}^{\mathrm{AM}}\), and \(c_{i,\boldsymbol\alpha}^{\Delta} = c_{D-i,\boldsymbol\alpha}^{\mathrm{CBJP}}\), \eqref{eq:CBJP-product-original-indexing} becomes
\begin{equation}
c_i(J_1\cdots J_s;R) = \sum_{|\boldsymbol\alpha|=i-1} \binom{i-1}{\boldsymbol\alpha} c_{i,\boldsymbol\alpha}^{\Delta} (J_1|\cdots|J_s;R).
\label{eq:CBJP-product-codimension-indexing}
\end{equation}

Recall that \(c_{i,\boldsymbol\alpha}^{(j)} (J_1|\cdots|J_s;M) := c^{(j)}\bigl( J_1^{[\alpha_1]},\ldots, J_j^{[\alpha_j+1]},\ldots, J_s^{[\alpha_s]};M \bigr)\). By Proposition~\ref{prop:diagonal-directional-sum}, \(c_{i,\boldsymbol\alpha}^{\Delta} (J_1|\cdots|J_s;M) = \sum_{j=1}^{s} c_{i,\boldsymbol\alpha}^{(j)} (J_1|\cdots|J_s;M)\). Thus, the module version of \eqref{eq:CBJP-product-codimension-indexing} rewrites as
\begin{equation}
c_i(J_1\cdots J_s;M) = \sum_{|\boldsymbol\alpha|=i-1} \binom{i-1}{\boldsymbol\alpha} \sum_{j=1}^{s} c_{i,\boldsymbol\alpha}^{(j)} (J_1|\cdots|J_s;M).
\label{eq:CBJP-product-directional-form}
\end{equation}
Hence, their product formula is precisely the sum of the directional product formula over all distinguished directions. Note that Theorem~\ref{thm:directional-product-formula-powers} provides the stronger multihomogeneous relation
\begin{equation}
c_i(J_1^{n_1}\cdots J_s^{n_s};M) = \sum_{|\boldsymbol\alpha|=i-1} \binom{i-1}{\boldsymbol\alpha} \mathbf n^{\boldsymbol\alpha} \sum_{j=1}^{s} n_j c_{i,\boldsymbol\alpha}^{(j)} (J_1|\cdots|J_s;M),
\label{eq:directional-power-comparison-CBJP}
\end{equation}
where \(\mathbf n^{\boldsymbol\alpha} =n_1^{\alpha_1}\cdots n_s^{\alpha_s}\). Setting \(n_1=\cdots=n_s=1\) in \eqref{eq:directional-power-comparison-CBJP} recovers \eqref{eq:CBJP-product-directional-form}. Consequently:
\begin{enumerate}[label=\textup{(\roman*)}]
\item For \(M=R\) and \(n_1=\cdots=n_s=1\), it recovers \cite[Theorem~5.3]{CallejasBedregalJorgePerez};
\item For finitely generated module \(M\), it extends the formula to module structures;
\item For \(n_1,\ldots,n_s\in\mathbb N_{>0}\), it yields a multihomogeneous power version;
\item The directional coefficients preserve the individual components aggregated by \(c_{i,\boldsymbol\alpha}^{\Delta}\).
\end{enumerate}

When \(J_1,\ldots,J_s\) are \(\mathfrak m\)-primary, Theorem~\ref{thm:directional-primary-compatibility} gives 
$$c_{D,\boldsymbol\alpha}^{(j)} (J_1|\cdots|J_s;M) = e\bigl( J_1^{[\alpha_1]},\ldots, J_j^{[\alpha_j+1]},\ldots, J_s^{[\alpha_s]};M \bigr)$$ for \(|\boldsymbol\alpha|=D-1\). Thus
\begin{equation}
c_{D,\boldsymbol\alpha}^{\Delta} (J_1|\cdots|J_s;M) = \sum_{j=1}^{s} c_{D,\boldsymbol\alpha}^{(j)} (J_1|\cdots|J_s;M) = \sum_{j=1}^{s} e\bigl( J_1^{[\alpha_1]},\ldots, J_j^{[\alpha_j+1]},\ldots, J_s^{[\alpha_s]};M \bigr),
\label{eq:primary-diagonal-comparison-CBJP}
\end{equation}
recovering \cite[Proposition~5.2(2b)]{CallejasBedregalJorgePerez} when \(M=R\). In addition, for \(\mathbf d\in\mathbb N^s\) with \(|\mathbf d|=D\), summing the combinatorial factors over direction \(j\) yields \(\sum_{j:d_j>0} \frac{(D-1)!} {d_1!\cdots(d_j-1)!\cdots d_s!} = \frac{(D-1)!}{d_1!\cdots d_s!} \sum_{j=1}^{s}d_j = \frac{D!}{d_1!\cdots d_s!}\). Therefore, the directional power formula specializes to the classical Bhattacharya--Rees--Teissier formula
\begin{equation}
e(J_1^{n_1}\cdots J_s^{n_s};M) = \sum_{\substack{\mathbf d\in\mathbb N^s\\|\mathbf d|=D}} \frac{D!}{d_1!\cdots d_s!}\, e\bigl( J_1^{[d_1]},\ldots,J_s^{[d_s]};M \bigr) \mathbf n^{\mathbf d}.
\label{eq:classical-product-from-directional-comparison}
\end{equation}
\end{remark}



\section{Acknowledgments}
This work was developed during the postdoctoral fellowship of the first author at the Instituto de Ci\^encias Matem\'aticas e de Computa\c{c}\~ao (ICMC), Universidade de S\~ao Paulo (USP) - 2026. He gratefully acknowledges financial support from FAPESP, grant 2025/20830-5. The first author also gratefully acknowledges the Universidade Tecnol\'ogica Federal do Paran\'a -- Campus Guarapuava for the opportunity to undertake his postdoctoral research. The second author thanks CNPq-Brazil, grant 304851/2025-6 and -FAPESP-Brazil, grant 2025/21618-0. OpenAI’s ChatGPT was used during the preparation of this work as an auxiliary generative-AI tool for  discussion, checks of algebraic manipulations and examples, comparison with cited literature, and assistance with English exposition and LaTeX. All material used in the final manuscript was checked by the authors against the relevant arguments and sources. The authors assume full responsibility for the final text.

\end{document}